\documentclass[a4paper,11pt,reqno, english]{amsart}

\usepackage{anysize,color}
\usepackage[utf8]{inputenc}
\usepackage[dvipsnames]{xcolor}
\usepackage[colorlinks=true,urlcolor=blue,linkcolor={magenta},citecolor={orange}]{hyperref}  
\usepackage{float}
\usepackage{amsfonts,amssymb,amsmath}
\usepackage{amsthm}    
\usepackage{url}
\usepackage{paralist}
\usepackage{tikz}
    \usetikzlibrary{decorations.markings,arrows,shapes.callouts}
\usepackage{rotating} 
\usepackage[mathscr]{euscript}
\usepackage{charter}
\usepackage{verbatim}
\usepackage{tikz-cd}
\usepackage{mathtools}
\usepackage{tabularx}
\usepackage{soul}
\usepackage[arrow,curve,matrix,tips,2cell]{xy}  
  \SelectTips{eu}{10} \UseTips
  \UseAllTwocells
 
 \usepackage{bbold}

\newtheorem*{theorem*}{Theorem}
\newtheorem{theorem}{Theorem}[section]

\newtheorem{lemma}[theorem]{Lemma} 
\newtheorem{problem}[theorem]{Problem}
\newtheorem{corollary}[theorem]{Corollary} 

\theoremstyle{definition}

\theoremstyle{remark}

\newcommand\R{\mathbb R}

\newcommand\Z{\mathbb Z}
\newcommand\N{\mathbb N}
\newcommand\Q{\mathbb Q}
\newcommand\FF{\mathbb F}

\newcommand\NN{\mathscr N}

\newcommand\U{\mathscr U}
\newcommand\CC{\mathscr C}

\newcommand\E{\mathbb{E}}
\newcommand\K{\mathscr{K}}
\newcommand\LL{\mathscr{L}}

\newcommand\pt{\mathrm{pt}}
 
\DeclareMathOperator{\vertex}{\mathrm{vert}}
\DeclareMathOperator{\spann}{\mathrm{span}}
\DeclareMathOperator{\im}{\mathrm{im}}

\DeclareMathOperator{\cone}{\mathrm{cone}}

\DeclareMathOperator{\sspan}{\mathrm{span}}

\DeclareMathOperator{\interior}{\mathrm{int}}
\DeclareMathOperator{\relint}{\mathrm{relint}}

\DeclareMathOperator{\conv}{\mathrm{conv}}

\DeclareMathOperator{\cl}{\mathrm{cl}}

\begin{document}

\dedicatory{Dedicated to the memory of Martin Aigner, a great man and an excellent mathematician}

\title[Colourful Carath\'eodory's Theorem plus a constraint]{Colourful Carath\'eodory's Theorem plus a constraint}

\author[Blagojevi\'c]{Pavle V. M. Blagojevi\'{c}}
\thanks{The research by Pavle V. M. Blagojevi\'{c} leading to these results has
received funding from the Serbian Ministry of Science, Technological development and Innovations. It was also supported by G\"unter Ziegler's group at the Freie Universit\"at Berlin ``Arbeitsgruppe Diskrete Geometrie und Topologische Kombinatorik''. My gratitude goes to Berlin Mathematical School for supporting my PhD students.}
\address{Inst. Math., FU Berlin, Arnimallee 2, 14195 Berlin, Germany\hfill\break
	\mbox{\hspace{4mm}}Mat. Institut SANU, Knez Mihailova 36, 11001 Beograd, Serbia}
\email{blagojevic@math.fu-berlin.de \& pavleb@mi.sanu.ac.rs}

\begin{abstract}
	We develop a topological framework in an attempt to generalise the classical colourful Carath\'eodory theorem by imposing an additional constraint.
	For that, we introduce the notion of zero-avoding complexes and covering criteria for the existence of colourful transversals.
	Using the developed method in combination with the homological Nerve theorem of Meshulam we recover all known versions of the colourful Carath\'eodory's theorem and prove a constraint extension which, in particular, implies an extension of the original (affine) Tverberg result.
\end{abstract}

\maketitle

\bigskip
\section{Introduction and statement of central results}
\bigskip

%
%
B\'ar\'any's classical colourful Carath\'eodory's theorem~\cite{Barany1982} from 1982 is one of the cornerstones of Combinatorial Convexity which, in particular, transforms Tverberg's seminal result~\cite{Tverberg1966}, via Sarkaria's tensor trick~\cite{Sarkaria1992}, into a mere corollary.
Furthermore, it yields multiple applications, like colourful versions of Fenchel's and Kirchberger's theorem~\cite{Barany2021}.
Important extensions of the colourful Carath\'eodory theorem were obtained only recently by Arocha, B\'ar\'any, Bracho, Fabila \& Montejano~\cite{ArochaBaranyBrachoFabilaMontejano2009}, Holmsen, Pach \& Tverberg~\cite{HolmsenPachTverberg2008}, and Kalai \& Meshulam~\cite{KalaiMeshulam2005}.
This turns the colourful Carath\'eodory theorems into fundamental and intriguing results.

\medskip
In the following, when working with a simplicial complex, for brevity, we always have in mind all three variants: the abstract simplicial complex (family of subsets of the set of vertices), the geometric simplicial complex (family of simplices in a Euclidean space), and finally the geometric realisation of simplicial complex (topological space).
Hence, from now on, in the statements of this paper the context determines which instance of a simplicial complex is assumed, allowing us to avoid unnecessarily complex notation yoga.
This should not cause any confusion.
More details on simplicial complexes can be found, for example, in the classical soucre~\cite{Munkres1984}.

\medskip
For integers $d\geq 1$ and $N\geq 0$, we denote by $\R^d$ a $d$-dimensional Euclidean affine space and by $\Delta_N$ an $N$-dimensional simplex.
For a simplicial complex $\K$ we denote its vertex set by $\vertex (\K)$.
On the vertex set $V:=\vertex(\K)$, and a subset of its vertices $U\subseteq V$ we define the {\em induced sub-complex} on $U$ by
\[
	\K[U]:=\{ F\subseteq U : F\in\K\}.
\]
In particular, $\K=\K[V]$.
Here we think of $\K$ as an abstract simplicial complex which allowed us to give a simple definition of an induced (abstract) sub-complex.
Additionally, if $r\geq 1$ is an integer, then we denote by $\K^{*r}$ the $r$-fold join and by $\K^{*r}_{\Delta(2)}$ the $r$-fold $2$-wise deleted join of the simplicial complex  $\K$.
In particular, $\K^{*1}\cong \K^{*1}_{\Delta(2)}\cong \K$.
For more insight in the definition and properties of deleted joins consult for example~\cite[Sec.\,6.3]{Matousek2008} and~\cite[Sec,\,3.2.2]{BlagojevicZiegler2017}.

\medskip
\subsection{The general colourful Carath\'eodory problem}
The main result of the classical 1982 work of B\'ar\'any~\cite[Thm.\,2.1]{Barany1981}, when translated into the language of affine maps, says the following.

\medskip
\begin{theorem}[The colourful Carath\'eodory I]
	\label{th : colourful_Caratheodory_01}
	Let $N\geq 1$ and $d\geq 1$ be integers, and assume that $A\colon \Delta_N\longrightarrow\R^d$ is an affine map.

	\smallskip\noindent
	If there exists an integer $r\geq d+1$ and $r$ pairwise disjoint non-empty faces $F_1,\dots ,F_{r}$ of the simplex $\Delta_N$ such that
	$
		0\in A(F_1)\cap\cdots\cap A(F_r)
	$,
	then there exists a selection of vertices $p_1\in F_1,\dots, p_r\in F_r$ with the property that $0\in A(J)$ where $J=\conv\{p_1,\dots,p_r\}\subseteq\Delta_N$ is the face of the simplex $\Delta_N$ spanned by the selected vertices  $p_1,\dots,p_r$.
\end{theorem}

\medskip
For an illustration of Theorem~\ref{th : colourful_Caratheodory_01} observe Figure ~\ref{fig c01} where $A(\vertex(F_1))$, $A(\vertex(F_2))$ and $A(\vertex(F_3))$ are indicated by different colors.

\begin{figure}[h]
	\includegraphics[scale=1.1]{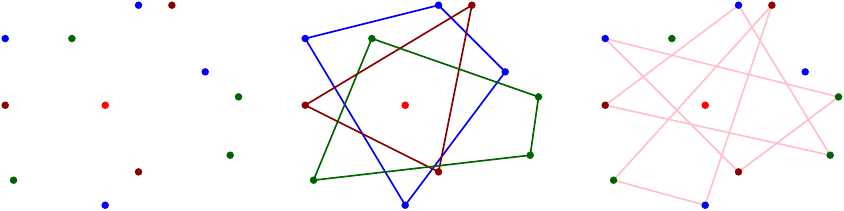}
	\caption{\small An illustration of Theorem~\ref{th : colourful_Caratheodory_01} where the red point is the zero in the plane.}
	\label{fig c01}
\end{figure}

\medskip
Already in the original paper of B\'ar\'any one encounters the first extension of the previous result, see~\cite[Thm.\,2.3]{Barany1981}.

\medskip
\begin{theorem}[The colourful Carath\'eodory II]
	\label{th : colourful_Caratheodory_02}
	Let $N\geq 1$ and $d\geq 1$ be integers, and let $A\colon \Delta_N\longrightarrow\R^d$ be an affine map.

	\smallskip\noindent
	If there exists an integer $r\geq d+1$ and $r$ pairwise disjoint non-empty faces $F_1,\dots, F_{r}$ of the simplex $\Delta_N$ such that
	$
		0\in A(F_1)\cap\cdots\cap A(F_{r-1})
	$,
	then for an arbitrary vertex $p\in F_r$ there exists a selection of vertices $p_1\in F_1,\dots, p_{r-1}\in F_{r-1}$ with the property that $0\in A(J)$ where $J=\conv\{p_1,\dots,p_{r-1},p\}\subseteq\Delta_N$ is the face of the simplex $\Delta_N$ spanned by the vertices $p_1,\dots,p_{r-1},p$.
\end{theorem}

\medskip
For the next extension we needed to wait almost three decades.
Independently two groups of authors, Arocha, B\'ar\'any, Bracho, Fabila \& Montejano~\cite[Thm.\,1]{ArochaBaranyBrachoFabilaMontejano2009} and Holmsen, Pach \& Tverberg~\cite[Thm.\,8]{HolmsenPachTverberg2008}, proved the following result.

\medskip
Here we will abuse notation and for faces $F_i$ and $F_j$ of a simplex $\Delta_N$ we denote by $F_i\cup F_j$ the minimal face of $\Delta_N$ containing both $F_i$ and $F_j$.
This means, $\vertex(F_i\cup F_j)=\vertex(F_i)\cup\vertex(F_j)$.

\medskip
\begin{theorem}[The colourful Carath\'eodory III]
	\label{th : colourful_Caratheodory_03}
	Let $N\geq 1$ be an integer, and let $A\colon \Delta_N\longrightarrow\R^d$ be an affine map.

	\smallskip\noindent
	If there exists an integer $r\geq d+1$ and $r$ pairwise disjoint non-empty faces $F_1,\dots, F_{r}$ of the simplex $\Delta_N$ such that
	$
		0\in \bigcap_{1\leq i<j\leq r} A(F_i\cup F_j)
	$,
	then there exists a selection of vertices $p_1\in F_1,\dots, p_r\in F_r$ with the property that $0\in A(J)$ where $J=\conv\{p_1,\dots,p_r\}\subseteq\Delta_N$ is the face of the simplex $\Delta_N$ spanned by the selected vertices  $p_1,\dots,p_r$.
\end{theorem}

\medskip
Yet another intriguing extension is due to Kalai \& Meshulam~\cite[Cor.\,1.4]{KalaiMeshulam2005} which came as a consequence of their topological colourful Helly theorem~\cite[Thm.\,1.6]{KalaiMeshulam2005}.
For the proper statement of the result we need to introduce the notion of a matroid.

\medskip
An abstract simplicial complex $\mathcal{M}$, on the set of vertices $M:=\vertex({\mathcal{M}})$, is called a {\em matroid} if for every two faces $F$ and $F'$ of $\mathcal{M}$ such that $\dim(F)<\dim(F')$, there exists a vertex $v\in F'-F$ such that $F\cup\{v\}$ is a face of $\mathcal{M}$.
To every matroid $\mathcal{M}$ we associate the {\em rank function} $\rho_{\mathcal{M}}\colon 2^{M}-\{\emptyset\}\longrightarrow\N$ by setting
\[
	\rho_{\mathcal{M}}(S):= \dim (\mathcal{M}[S])+1
\]
for every $\emptyset\neq S\subseteq M$.
Recall, here $\mathcal{M}[S]$ denotes the induced sub-complex of  $\mathcal{M}$ on $S$.

\medskip
As an illustration
of a matroid we give a relevant example: Let $k\geq 1$ be an integer and let $M_1,\dots,M_k$ be a collection of $0$-dimensional simplicial complexes.
The join simplicial complex $M_1*\dots*M_k$ is the so-called partition matroid, see Figure~\ref{fig c02}.

\medskip
Now, the result of Kalai \& Meshulam~\cite[Cor.\,1.4]{KalaiMeshulam2005} can be formulated as follows.

\medskip
\begin{theorem}[The colourful Carath\'eodory IV]
	\label{th : colourful_Caratheodory_04}
	Let $N\geq 1$ be an integer, and let $A\colon \Delta_N\longrightarrow\R^d$ be an affine map.
	Suppose that the matroid $\mathcal{M}$ is a sub-complex of the simplex $\Delta_N$ with $M=\vertex(\mathcal{M})=\vertex(\Delta_N)$.

	\smallskip\noindent
	If for every face $S$ of the simplex $\Delta_N$ with the property that $\rho_{\mathcal{M}}(S)=\rho_{\mathcal{M}}(M)$ and $\rho_{\mathcal{M}}(M-S)\leq d$ it holds that $0\in A(S)$, then there exists a face $J$ of the matroid $\mathcal{M}$ such that $0\in A(J)$.
\end{theorem}

\begin{figure}[h]
	\includegraphics[scale=1.6]{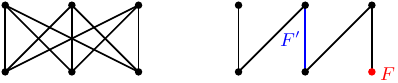}
	\caption{\small An illustration of the partition matroid $[3]*[3]$ and its sub-complex which is not a matroid since, for example, faces $F'$ and $F$ does not satisfy the defining property of a matroid.}
	\label{fig c02}
\end{figure}

\medskip
Summarizing the previous results we see that we always begin with an affine map $A\colon \Delta_N\longrightarrow\R^d$ from a simplex into some Euclidean space, and then we fix two sets of data
\begin{compactitem}[\quad --]
	\item a simplical sub-complex $\mathcal{K}$ of $\Delta_N$, with the property that $\vertex(\mathcal{K})=\vertex(\Delta_N)$, and
	\item a family of faces $\mathcal{S}$ of the simplex $\Delta_N$ with the property that for every $S\in \mathcal{S}$ we have $0\in A(S)$.
\end{compactitem}
Note that the second set of data also restricts the class of affine maps which can be considered.
Then our goal is to find a face $J$ of $\mathcal{K}$ with the property that $0\in A(J)$.

\medskip
In this language the previous results correspond to the following setups
\begin{compactitem}[\quad --]
	\item[\quad -- Theorem~\ref{th : colourful_Caratheodory_01}:]
	$\vertex(\Delta_N)=\vertex(F_1)\sqcup\cdots\sqcup \vertex(F_{r})$, $\mathcal{S}=\{F_1,\dots,F_r\}$, and $\mathcal{K}=\vertex(F_1)*\cdots * \vertex(F_{r})$,
	\item[\quad-- Theorem~\ref{th : colourful_Caratheodory_02}:]
	$\vertex(\Delta_N)=\vertex(F_1)\sqcup\cdots\sqcup \vertex(F_{r-1})\sqcup\{p\}$, $\mathcal{S}=\{F_1,\dots,F_{r-1}\}$, and $\mathcal{K}=\vertex(F_1)*\cdots * \vertex(F_{r-1})*\{p\}$,
	\item[\quad--  Theorem~\ref{th : colourful_Caratheodory_03}:]
	$\vertex(\Delta_N)=\vertex(F_1)\sqcup\cdots\sqcup \vertex(F_{r})$, $\mathcal{S}=\{F_i\cup F_j : 1\leq i<j\leq r\}$, and $\mathcal{K}=\vertex(F_1)*\cdots * \vertex(F_{r})$,
	\item[\quad-- Theorem~\ref{th : colourful_Caratheodory_04}:] $\mathcal{K}=\mathcal{M}$ is a matroid with $\vertex(\mathcal{M})=\vertex(\Delta_N)$ and $\mathcal{S}=\{ S\subseteq\Delta_N : \rho_{\mathcal{M}}(S)=\rho_{\mathcal{M}}(M),\,\rho_{\mathcal{M}}(M-S)\leq d\}$.
\end{compactitem}
Thus, the general colourful Carath\'eodory problem can be formulated as follows.

\medskip
\begin{problem}[The general colourful Carath\'eodory problem]\label{prob : general_colourful_Caratheodory theorem}
Let $N\geq 1$ and $d\geq 1$ be integers.
Find all simplicial sub-complexes $\mathcal{K}$ of the simplex $\Delta_N$, with $\vertex(\mathcal{K})=\vertex(\Delta_N)$, and all families $\mathcal{S}$ of faces of $\Delta_N$ such that the following implication holds: If an affine map $A\colon\Delta_N\longrightarrow\R^d$ satisfies that $0\in A(S)$ for every face $S\in \mathcal{S}$, then there exists a face $J$ of $\mathcal{K}$ with $0\in A(J)$.
\end{problem}

In this paper, building on the ideas of Kalai \& Meshulam~\cite{KalaiMeshulam2005}, we develop an algorithm for solving Problem~\ref{prob : general_colourful_Caratheodory theorem}.
Then, using this method and the homological Nerve theorem of Meshulam~\cite{Meshulam2001} we reprove all known results and also prove the following new generalisation of the colourful Carath\'eodory theorem.
As a consequence we obtain the corresponding extension of the classical Tverberg theorem.

\medskip
\subsection{A constrained colourful Carath\'eodory theorem}
The central result of this paper is the following constrained colourful Carath\'eodory theorem.

\medskip
\begin{theorem}\label{th : main_result_2}
	Let $N\geq 1$, $d\geq 2$ and $r\geq d+1$ be integers, and let $A\colon \Delta_N\longrightarrow\R^d$ be an affine map.
	Suppose that
	\begin{compactenum}[\rm (1)]
		\item $F_1,\dots, F_{r}$ is a collection of pairwise disjoint non-empty faces of the simplex $\Delta_N$ such that
		$
			0\in A(F_1)\cap\cdots\cap A(F_r)
		$, and
		\item $\mathcal{K}$ is a simplicial sub-complex  $L*\vertex(F_3)*\cdots *\vertex(F_r)$ of $\Delta_N$ where $L$ is a path-connected sub-complex of the join complex $\vertex(F_1)*\vertex(F_2)$ with $\vertex(L)=\vertex(F_1)\cup\vertex(F_2)$.
	\end{compactenum}
	Then there exists a face $J$ of $\mathcal{K}$ with $0\in A(J)$.
\end{theorem}

\medskip
The example in Figure~\ref{fig c02} indicates that Theorem~\ref{th : main_result_2} is not a consequence of the matroid version of the colourful Carath\'eodory theorem~\ref{th : colourful_Caratheodory_03}.
Further, an illustration of Theorem~\ref{th : main_result_2} in Figure~\ref{fig c03} indicates that additional constraints on the simplicial complex $\K$ are not possible, suggesting that our theorem cannot be generalized further.

\begin{figure}[h]
	\includegraphics[scale=1.2]{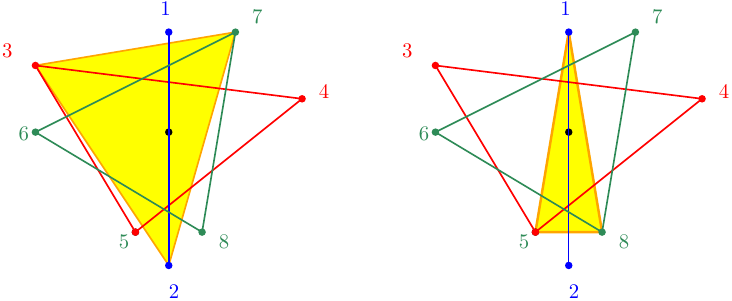}
	\caption{\small An instance of Theorem~\ref{th : main_result_2} for $\Delta_7\longrightarrow\R^2$ with $F_1=\conv(\{1,2\}),F_2=\conv(\{3,4,5\}),F_3=\conv(\{6,7,8\})$ and $L\subseteq\{3,4,5\}*\{6,7,8\}$ determined by the edges $\conv(\{3,6\}), \conv(\{3,7\}), \conv(\{4,7\}), \conv(\{4,8\}), \conv(\{5,8\})$.}
	\label{fig c03}
\end{figure}

\medskip
The previous result, as a direct consequence, yields a strengthening of 1966 Tverberg's theorem via Sarkaria's tensor trick~\cite{Sarkaria1992}; see also~\cite[Ch.\,13]{Barany2021}.
Until now the only known strengthening of the classical Tverberg result is the Optimal colored Tverberg theorem of Blagojevi\'c, Matschke \& Ziegler~\cite{Blagojevic2009,Blagojevic2011-2} which holds only when the overlapping parameter is a prime.
The following new result has no restriction on the overlapping parameter.

\medskip
\begin{corollary}
	\label{th : main_result_1}
	Let $r\geq 2$, $d\geq 1$ and $N\geq (d+1)(r-1)$	be integers, and let $a\colon \Delta_N\longrightarrow\R^d$ be an affine map.
	Consider an arbitrary sub-complex $L$ of the $1$-dimensional simplicial complex $[r]*[r]$ that is path-connected and $\vertex(L)=\vertex([r]*[r])$.
	Let $\LL$ be the sub-complex $L*\big([r]^{* N-1}\big)$ of the deleted join $(\Delta_N)^{*r}_{\Delta(2)}\cong [r]^{* N+1}$.

	\smallskip\noindent
	Then there are $r$ pairwise disjoint faces $F_1,\dots, F_r$ of the simplex $\Delta_N$ with the property that the face $F_1*\dots*F_r\in\LL$ and that $a(F_1)\cap\cdots\cap a(F_r)\neq\emptyset$.
\end{corollary}

\medskip
While Theorem~\ref{th : main_result_2} is the central result of this paper, an essential contribution of this work lies in development of new methods and tools which compose into an algorithm for advancements on the general colourful Carath\'eodory problems and consequently on general (affine) Tverberg problems as well as on the B\'ar\'any--Kalai conjecture~\cite[Conj.\,4.5]{BaranyKalai2022}.

\medskip
In the following we will prove all versions of the colourful Carath\'eodory theorem except the matroid version, Theorem~\ref{th : colourful_Caratheodory_04}, because the relevant proof is content of the seminal work of Kalai \& Meshulam~\cite[Cor.\,1.4]{KalaiMeshulam2005}.

\medskip
We start with an introduction and a study of zero-avoiding complexes, the associated Alexander duals and the corresponding induced sub-complexes, Section~\ref{sec:ZeroAvoidingComplexesAlexanderDualsInducedSubcomplexes}.
Next, we relate the existence of a solution to Problem~\ref{prob : general_colourful_Caratheodory theorem} with the claim that a particular collection of sub-complexes covers a particularly constructed simplicial complex and then prove Theorem~\ref{th : colourful_Caratheodory_01}, Theorem~\ref{th : colourful_Caratheodory_02}, Theorem~\ref{th : colourful_Caratheodory_03} and Theorem~\ref{th : main_result_2} in Section~\ref{sec : Covering scheme and proofs of colourful Caratheodory theorems}.

\medskip
Finally, using the notion of vector-avoiding complexes in the place of zero-avoiding complexes we prove in \cite{BlagojevicKarasev2025} the cone analogues of all known versions of the colourful Carath\'eodory theorem.

\subsection*{Acknowledgements}

I would like to thank my students Nevena Pali\'c and Nikola Sadovek for many useful discussions and corrections. Furthermore, my gratitude goes to Micheal Crabb, for offering his expertise and ideas as well as many essential insights.
Matija Blagojevi\'c helped with the final preparation of the manuscript.
I am grateful to Roman Karasev for prompt suggestions and corrections of the first arXiv version of this manuscript.

\bigskip
\section{Zero-avoiding complexes, its Alexander duals and induced sub-complexes}
\label{sec:ZeroAvoidingComplexesAlexanderDualsInducedSubcomplexes}
\bigskip

We begin by introducing the notion of a zero-avoiding simplicial complex associated to an affine map from a simplex into an affine Euclidean space.
These complexes were recently studied, with a different focus, by Bludov in his beautiful work \cite{Bludov2025} where the zero-avoiding simplicial complex is called simplicial complex of non-balanced subsets.

\medskip
\subsection{Definition of zero-avoiding complexes}
\label{subsec : Definition of zero-avoiding complexes}

Let $\E$ be a finite dimensional affine Euclidean space with the fixed structure of a real vector space.
In particular, we know what is $0\in \E$.
The standard scalar product on the vector space structure of $\E$ is denoted by $\langle\cdot,\cdot\rangle$.

\medskip
Let $V$ be a finite set (in $\R^{\infty}=\bigoplus_{n\geq 1}\R^n$) and $\Delta_V$ the simplex with the vertex set $V$.
Without loss of generality we can assume that $V:=\{e_1,\dots, e_k\}$ for $k:=|V|\in\N$ and $e_i$ is the $i$-th vector of the standard basis $(e_i : i\in\N)$ of  $\R^{\infty}$.
Fix an affine map $A\colon \Delta_V\longrightarrow\E$ with the property that  $A(V)\subseteq \E-\{0\}$, or in other words, such that no vertex is mapped to zero.
An (abstract) simplicial complex on the set of vertices $V$ defined by
\begin{equation}\label{def of zero-avoiding complex}
	\CC_A:=\{U\subseteq V : 0\notin \conv(A(U))=A(\conv(U))\}
\end{equation}
is called the {\em zero-avoiding complex} of the affine map $A$.
The set $V$ is indeed the set of vertices of $\CC_V$ because $0\notin A(V)$.
Furthermore, $0\notin \conv(A(U))$ implies that $0\notin \conv(A(U'))$ for every subset $U'\subseteq U$, so $\CC_A$ is  an abstract simplicial complex as announced.

\medskip
For example, let $\E\cong \R$.
Consider $V_1=\{e_1,e_2\}$ and $V_2=\{e_1,e_2,\dots, e_k\}$ with affine maps $A_1\colon\Delta_{V_1}\longrightarrow\E$, $A_1(e_1)=-1$, $A_1(e_2)=1$ and $A_2\colon\Delta_{V_2}\longrightarrow\E$, $A_2(e_i)=i$ for $0\leq i\leq s$.
Then $\CC_{A_1}=[2]$ is the two point space and $\CC_{A_2}=\Delta_{|V_2|}$ is a $(k-1)$-dimensional simplex.

\medskip
Now consider the ``normalisation'' of an affine map $A\colon \Delta_V\longrightarrow\E$  onto the unit sphere $S(\E)$, that is the affine map $\overline{A}\colon \Delta_V\longrightarrow\E$ defined, on the vertices, by
\[
	\overline{A}(v):=\frac{A(v)}{\|A(v)\|} \ \in \ S(\E)
\]
for all $v\in V$.
It is important to notice that the zero-avoiding complexes of $A$ and $\overline{A}$ coincide, or in other words $\CC_A=\CC_{\overline{A}}$.
Indeed, for any collection of points $u_1,\dots, u_s\in \E-\{0\}$ it holds that:
\[
	0\in \conv\{u_1,\dots, u_s\}
	\qquad\text{if and only if}\qquad
	0\in \conv\Big\{\frac{u_1}{\|u_1\|},\dots, \frac{u_s}{\|u_s\|}\Big\}.
\]
Consequently, without loss of generality, from now on, we can assume that $A(V)\subseteq S(\E)$.

\medskip
\subsection{Families of open and closed hemispheres}
\label{subsec : Families of open and closed hemispheres}
In order to understand the topology of the complex $\CC_A$ we make a detour and first study families of open and closed hemispheres on the unit sphere $S(\E)$ and their unions and intersections.

\medskip
For a unit vector $u\in S(\E)$ we denote the associated open and closed hemisphere of $S(\E)$ by
\[
	\mathscr{B}_u:=\{x\in S(\E) : \langle u,x\rangle >0\}
	\qquad\text{and}\qquad
	\mathscr{D}_u:=\{x\in S(\E) : \langle u,x\rangle \geq 0\},
\]
respectively.
Evidently, $\mathscr{B}_u\subseteq \mathscr{D}_u$.
If $U\subseteq S(\E)$ is a non-empty finite collection of unit vectors in $\E$ then we set
\[
	\mathscr{B}(U):=\bigcap_{u\in U}\mathscr{B}_u,\qquad
	\mathscr{D}(U):=\bigcap_{u\in U}\mathscr{D}_u
	\qquad\text{and}\qquad
	\mathscr{B}\langle U\rangle :=\bigcup_{u\in U}\mathscr{B}_u,\qquad
	\mathscr{D}\langle U\rangle :=\bigcup_{u\in U}\mathscr{D}_u.
\]
It is evident that each $\mathscr{B}(U)$ is either an open $(m-1)$-dimensional ball or empty.
On the other hand nothing can be said a priori about $\mathscr{B}\langle U\rangle$ besides that it is a non-empty open subset of the sphere $S(\E)$.
Therefore, we mostly focus our attention on understanding the topology of the unions.

\medskip
The first non-elementary observation gives the criterion for non-emptiness of $\mathscr{B}(U)$.

\medskip
\begin{lemma}
	\label{lem : non-empty condition}
	Let $U\subseteq S(\E)$ be a finite collection of points.
	Then $\mathscr{B}(U)\neq\emptyset$ if and only if $0\notin\conv(U)$.
\end{lemma}
\begin{proof}
	Let $x\in \mathscr{B}(U)\neq\emptyset$, or in other words $\langle x,u\rangle >0$ for all $u\in U$.
	Consider the open half-space
	$
		H^+=\{ y\in \E : \langle x,y\rangle>0\}
	$
	determined by $x$.
	Then $U\subseteq H^+$ and so $\conv(U)\subseteq \conv(H^+)=H^+$.
	Since $0\notin H^+$ it follows that  $0\notin\conv(U)$ as desired.

	\medskip
	Conversely, suppose that $0\notin\conv(U)$.
	Then there exists a hyperplane $H=\{ y\in \E : \langle x_0,y\rangle=0\}$, for some $x_0\in S(\E)$, with the property that  $\conv(U)\subseteq H^+=\{ y\in \E : \langle x_0,y\rangle>0\}$.
	Since $U\subseteq \conv(U)$ it follows that $\langle x_0,u\rangle>0$ for every $u\in U$.
	Consequently, $x_0\in \bigcap_{u\in U}\mathscr{B}_u=\mathscr{B}(U)$, and so $\mathscr{B}(U)\neq\emptyset$, as claimed.

	Note that for $x_0$ we can take the vector $
		\frac{\overrightarrow{0\,p_{\conv(U)}(0)}}{\|\overrightarrow{0\,p_{\conv(U)}(0)}\|}$,
	where $p_{\conv(U)}$ denotes the metric projection of the convex body $\conv(U)$.
	For a definition and properties of the metric projection, consult for example~\cite[Sec.\,4.1]{Gruber2007}.
\end{proof}

\medskip
In the case of families of closed hemispheres the situation is a bit more delicate as we will see in the following lemmas.

\medskip
\begin{lemma}
	\label{lem : non-empty condition for closed}
	Let $U\subseteq S(\E)$ be a finite collection of points.
	Then $ \mathscr{D}(U)\neq\emptyset$ if and only if there exists a vector $x\in S(\E)$ such that $U\subseteq \mathscr{D}_x$.
\end{lemma}
\begin{proof}
	Let $x\in \mathscr{D}(U)\neq\emptyset$.
	Then for every $u\in U$ holds $\langle x,u\rangle\geq 0$.
	Hence, $U\subseteq \{ y\in S(\E) : \langle x, y\rangle\geq 0\}=\mathscr{D}_x$, as claimed.

	\medskip
	Suppose that $U\subseteq \mathscr{D}_x=\{ y\in S(\E) : \langle x, y\rangle\geq 0\}$ for some $x\in S(\E)$.
	Then $\langle x, u\rangle\geq 0$  for every $u\in U$ implying that $x\in \mathscr{D}(U)$.
	Hence, $\mathscr{D}(U)\neq\emptyset$.
\end{proof}

\medskip
Unlike in the case of the open hemisphere, the intersection of closed hemispheres is not always empty or contractible.
For the complete description we recall the notion of Robinson (spherical) convexity~\cite[Sec.\,9.1.\,(R)]{DanzerEtAl1963} and immediately obtain the following claim.

\begin{lemma}
	\label{lem : Robinson convex}
	Let $U\subseteq S(\E)$ be a finite collection of points, and assume that $\mathscr{D}(U)\neq\emptyset$.
	Then $\mathscr{D}(U)$ is Robinson convex and so either contractible or equal to an equatorial sphere $S^{\ell}$ for some $0\leq \ell\leq \dim(\E)-2$.
	In particular, if $U$ does not contain any pair of antipodal points, in other words if $U\cap (-U)=\emptyset$, then $\mathscr{D}(U)$ is contractible.
	Here, $-U:=\{-u : u\in U\}$.

\end{lemma}

\medskip
Switching our attention to the topology of the unions of families of open hemispheres we start with a criterion for $\mathscr{B}\langle U\rangle$ to cover the sphere $S(\E)$.

\medskip
\begin{lemma}
	\label{lem : covering of the sphere}
	Let $U\subseteq S(\E)$ be a finite collection of points.
	Then $\mathscr{B}\langle U\rangle =S(\E)$ if and only if $0\in\interior (\conv(U))$.
\end{lemma}
\begin{proof}
	We prove the left-to-right implication via contraposition.
	In other words, we prove the implication:
	\[
		0\notin\interior ( \conv(U) )\ \Longrightarrow \ \mathscr{B}\langle U\rangle\neq S(\E).
	\]
	Suppose $0\notin\interior (\conv(U))$.
	Then either
	\[
		\interior (\conv(U))=\emptyset\qquad\text{or}\qquad \interior(\conv(U))\subseteq H^+=\{ y\in \E : \langle x_0,y\rangle>0\}
	\]
	for some $x_0\in S(\E)$.
	In both cases, by choosing $x_0$ appropriately, we can guarantee that
	\[
		\emptyset\neq \conv(U) \ \subseteq \ \cl (H^+)=\{ y\in \E : \langle x_0,y\rangle\geq 0\}.
	\]
	Consequently, for every $u\in U\subseteq\conv(U)$ we have $\langle -x_0,u\rangle\leq 0$.
	In other words, $-x_0\notin\bigcup_{u\in U}\mathscr{B}_u=\mathscr{B}\langle U\rangle$, which implies that $\mathscr{B}\langle U\rangle\neq S(\E)$.

	\medskip
	Conversely, assume that $0\in\interior ( \conv(U))$.
	Then, according to~\cite[Thm.\,6]{McMullenShephard}, for every $u\in U$ there exists $t_u>0$ with the property that
	\[
		\sum_{u\in U}t_u\cdot u=0
		\qquad\text{and}\qquad
		\sum_{u\in U}t_u =1.
	\]
	Take an arbitrary point $y\in S(\E)$, and compute
	\[
		0=\langle 0\,,\,y\rangle=\big\langle\sum_{u\in U}t_u\cdot u\,,\,y\big\rangle= \sum_{u\in U}t_u\cdot\langle u,y\rangle .
	\]
	Now, either $\langle u,y\rangle=0$ for all $u\in U$, or there exists at least one  $u_0\in U$ with the property that $\langle u_0,y\rangle>0$.
	Since, $0\in\interior (\conv(U))$ we have that $U$ positively spans $\E$ (as a vector space) and so indeed there exists some $u_0\in U$ such that $\langle u_0,y\rangle>0$.
	Hence, $y\in \mathscr{B}_{u_0}\subseteq \mathscr{B}\langle U\rangle$ implying that $S(\E)=\mathscr{B}\langle U\rangle$.
\end{proof}

\medskip
From the previous two lemmas it is apparent that if $\mathscr{B}\langle U\rangle =S(\E)$, then $\mathscr{B}(U)=\emptyset$.
Furthermore, the previous lemma suggests that we should next discuss the topology of $\mathscr{B}\langle U\rangle $ in the case when $0\notin\interior (\conv(U))$.
This will be done in several steps starting with an extension of the previous lemma to the case when $0\in \relint (\conv (U))$.

\medskip
\begin{lemma}
	\label{lem :  case 01}
	Let $U\subseteq S(\E)$ be a finite collection of points.
	If $0\in \relint (\conv (U))$, then $S(\E)=\mathscr{B}\langle U\rangle \cup S$ and $\mathscr{B}\langle U\rangle\cap S=\emptyset$, where
	$
		S=\{y\in S(\E) : (\forall u\in U)\ \langle y,u\rangle=0 \}
	$
	is the unit sphere in the orthogonal complement of $\sspan(U)$.
\end{lemma}
\begin{proof}
	Let $x\in S(\E)-S$.
	Then, there exists $u_0\in U$ with the property that 	$\langle u_0,x\rangle\neq 0$.
	Since $0\in \relint(\conv (U))$ then for every $u\in U$ there is $t_u>0$ such that
	\[
		\sum_{u\in U}t_u\cdot u=0
		\qquad\text{and}\qquad
		\sum_{u\in U}t_u =1,
	\]
	consult~\cite[Thm.\,6]{McMullenShephard}.
	Like in the proof of the previous lemma we get that
	\[
		0=\langle 0\,,\,x\rangle=\big\langle\sum_{u\in U}t_u\cdot u\,,\,x\big\rangle= \sum_{u\in U}t_u\,\langle u,x\rangle .
	\]
	Since $\langle u_0,x\rangle\neq 0$ there is $u_1\in U$ so that  $\langle u_1,x\rangle > 0$.
	Hence, $x\in \mathscr{B}_{u_1}\subseteq \mathscr{B}\langle U\rangle$.
\end{proof}

\medskip
Next we consider the case when $0\in \conv (U) - \relint(\conv (U))$.

\medskip
\begin{lemma}
	\label{lem :  case 02}
	Let $U\subseteq S(\E)$ be a finite collection of points.
	If $0\in \conv (U) - \relint(\conv (U))$, then $\mathscr{B}\langle U\rangle$ is a contractible space.
\end{lemma}
\begin{proof}
	Let $0\in \conv (U) - \relint(\conv (U))$.
	Then there exists a unique maximal proper subset $U'\subsetneq  U$ with the property $0\in \relint\conv (U')$.
	In particular, $\conv (U')$ is a face of the polytope $\conv (U)$.
	Take $H:=\{y\in \E : \langle y,z\rangle=0\}$ to be a support hyperplane of $\conv (U)$ such that $\conv (U')=\conv (U)\cap H$ and $\conv (U)\subseteq H^+:=\{y\in \E : \langle y,z\rangle\geq 0\}$.
	The maximality of $U'$ implies that $U'=U\cap H$.
	Now we have that:
	\[
		(\forall u\in U')\ \langle u,z\rangle =0
		\qquad\text{and}\qquad
		(\forall u\in U-U')\ \langle u,z\rangle >0.
	\]
	Hence, $\langle u,-z\rangle \leq 0$ for every $u\in U$, or in other words, $-z\notin \mathscr{B}\langle U\rangle$.
	On the other hand, since $U-U'\neq\emptyset$, we have that $z\in \mathscr{B}\langle U\rangle$.

	\medskip
	To complete the proof of the lemma we define a strong deformation retraction from $\mathscr{B}\langle U\rangle$ to $\{z\}$ by
	\[
		h\colon \mathscr{B}\langle U\rangle\times [0,1]\longrightarrow \mathscr{B}\langle U\rangle,
		\qquad\qquad
		h(x,s):=\frac{(1-s)\cdot x+s\cdot z}{\| (1-s)\cdot x+s\cdot z\|}.
	\]
	Assuming $h$ is well-defined, it is evident that the map $h(\cdot,0)\colon \mathscr{B}\langle U\rangle\longrightarrow \mathscr{B}\langle U\rangle$ is the identity map and the map $h(\cdot,1)\colon \mathscr{B}\langle U\rangle\longrightarrow \mathscr{B}\langle U\rangle$ is the constant map with the value $z$.
	Hence, it remains to verify that $h$ is well-defined.

	\medskip
	First, we show that the denominator $\|(1-s)\cdot x+s\cdot z\|$ does not vanish.
	Assume contrary, there exist $s\in [0,1]$ and $x\in  \mathscr{B}\langle U\rangle$ such that $(1-s)\cdot x+s\cdot z=0$, or equivalently $(1-s)\cdot x= s\cdot (-z)$.
	The points $x$ and $-z$ are on the sphere $S(\E)$ and so $s\neq 0$ and $s\neq 1$, or in other words $s\in (0,1)$.
	Applying the norm to the equality $(1-s)\cdot x= s\cdot (-z)$ we get that $1-s=s$, and so $s=\frac12$.
	Consequently, $x=-z$.
	This cannot be because $x\in \mathscr{B}\langle U\rangle$ and we have shown that $-z\notin \mathscr{B}\langle U\rangle$.

	\medskip
	Finally, we verify that $\im(h)\subseteq \mathscr{B}\langle U\rangle$.
	In the case $s=0$ or $s=1$ this was already verified.
	For $s\in (0,1)$ we consider the following two cases:
	\begin{compactenum}[\rm\qquad (a)]

		\item Let $x\in \mathscr{B}_u$ for some $u\in U-U'$. Then
		\[
			\langle h(x,s),u\rangle=
			\frac{1-s}{\| (1-s)\cdot x+s\cdot z\|}\,\langle x,u\rangle+
			\frac{s}{\| (1-s)\cdot x+s\cdot z\|}\,\langle z,u\rangle \ > \ 0,
		\]
		because $x\in \mathscr{B}_u$ implies $\langle x,u\rangle> 0$ and $u\in U-U'$ implies $\langle z,u\rangle> 0$.
		Thus, $h(x,s)\in \mathscr{B}_u\subseteq \mathscr{B}\langle U\rangle$.

		\item Let $x\in \mathscr{B}_u$ for some $u\in U'$. Then again
		\[
			\langle h(x,s),u\rangle=
			\frac{1-s}{\| (1-s)\cdot x+s\cdot z\|}\,\langle x,u\rangle+
			\frac{s}{\| (1-s)\cdot x+s\cdot z\|}\,\langle z,u\rangle \ > \ 0,
		\]
		because $x\in \mathscr{B}_u$ implies $\langle x,u\rangle> 0$, $u\in U'$ implies $\langle z,u\rangle= 0$, and finally $s\neq 1$.
		Thus, again $h(x,s)\in \mathscr{B}_u\subseteq \mathscr{B}\langle U\rangle$.
	\end{compactenum}

	\medskip
	Therefore, the strong deformation retraction $H$ is well-defined, and we have completed the proof.
\end{proof}

\medskip
In the final case we analyze topology of $\mathscr{B}\langle U\rangle$ when $0\notin  \conv (U)$.

\medskip
\begin{lemma}
	\label{lem :  case 03}
	Let $U\subseteq S(\E)$ be a finite collection of points.
	If $0\notin  \conv (U)$, then $\mathscr{B}\langle U\rangle$ is a contractible space.
\end{lemma}
\begin{proof}
	Let $U\subseteq S(\E)$ be a finite collection of points and suppose that $0\notin  \conv (U)$.
	Consider the family $\U=\{\mathscr{B}_u : u\in U\}$ as an open and good cover of the  $\mathscr{B}\langle U\rangle=\bigcup_{u\in U}\mathscr{B}_u$, this means that every non-empty intersection is contractible.
	Then, by the Nerve theorem~\cite[Cor.\,4G.3]{Hatcher2002}, the space  $\mathscr{B}\langle U\rangle$ is homotopy equivalent to the nerve $\mathcal{N}_{\U}$ of the covering $\U$.
	Since, $0\notin  \conv (U)$ it follows that $\mathscr{B}(U) =\bigcap_{u\in U}\mathscr{B}_u\neq\emptyset$, see Lemma~\ref{lem : non-empty condition}.
	Therefore, the nerve $\mathcal{N}_{\U}$ is isomorphic to the $(|U|-1)$-dimensional simplex.
	Hence, $\mathscr{B}\langle U\rangle\simeq \mathcal{N}_{\U} \cong \Delta_{|U|-1}\simeq\pt$, as claimed.

\end{proof}

\medskip
In summary, we have proved that for $U\subseteq S(E)$ it holds that:
\[
	\mathscr{B}(U)\ \cong \
	\begin{cases}
		\ \cong D^{\dim(\E)-1}, & 0\notin\conv(U), \\
		\ = \emptyset,          & 0\in\conv(U),
	\end{cases}
\]
and
\begin{equation}
	\label{eq : 00}
	\mathscr{B}\langle U\rangle\
	\begin{cases}
		\ =S(\E),                                                           & 0\in\interior (\conv(U)),            \\
		\ =S(\E)-\{y\in S(\E) : (\forall u\in U)\ \langle y,u\rangle=0 \} , & 0\in \relint (\conv (U)),            \\
		\ \simeq \pt,                                                       & 0\in \conv (U) - \relint(\conv (U)), \\
		\ \simeq \pt,                                                       & 0\notin  \conv (U).
	\end{cases}
\end{equation}
Here $\pt$ stands for one point topological space.

\medskip
\subsection{Zero-avoiding complex as a nerve}

In this section we connect definition of zero-avoiding complex and the findings of the previous section with a hope to say more about the topology of zero-avoiding complexes.

\medskip
Let $V$ be a finite set (in $\R^{\infty}$) and let $\Delta_V$ be the simplex with the vertex set $V$.
Fix an affine map $A\colon \Delta_V\longrightarrow\E$ from the simplex $\Delta_V$ into a finite dimensional Euclidean space $\E$ such that $A(V)\subseteq S(\E)$.
Consider the family of open hemispheres
\[\U_A:=\{\mathscr{B}_{A(v)} \subseteq S(\E) : v \in V\}\]
induced by the affine map $A$.
In particular,  $\U_{A(V)}$ is a good cover of $\mathscr{B}\langle A(V)\rangle$, that is all intersections of elements of  $\U_{A(V)}$ are either empty of contractible.

\medskip
First, we relate the nerve of the cover $\U_A$ with the zero-avoiding complex $\CC_A$.

\medskip
\begin{lemma}
	\label{lem : nerve}
	Let $V$ be a finite set and let $A\colon \Delta_V\longrightarrow\E$ be an affine map such that $A(V)\subseteq S(\E)$.
	Then  $\CC_{A}$ is isomorphic, as a simplicial complex, to the nerve $\NN_{\U_A}$ of the open cover $\U_A$.
\end{lemma}
\begin{proof}
	The claim follows directly from the definition of the nerve and Lemma~\ref{lem : non-empty condition}.
	Indeed, for $U\subseteq V$ the following equivalences hold:
	\begin{align*}
		U\text{ spans a simplex in the nerve of }\U_A \quad & \Longleftrightarrow	\quad \mathscr{B}( A(U))\neq \emptyset            \\
		                                                    & \Longleftrightarrow	\quad 0\notin\conv(A(U))                          \\
		                                                    & \Longleftrightarrow	\quad A(U) \text{ spans a simplex in }\CC_{A(V)}.
	\end{align*}
\end{proof}

\medskip
Now, using  Lemma~\ref{lem :  case 01}, Lemma~\ref{lem :  case 02}, Lemma~\ref{lem :  case 03}, Lemma~\ref{lem : nerve}, and the Nerve theorem~\cite[Cor.\,4G.3]{Hatcher2002} we get the following characterisation of a zero-avoiding complex.
The next lemma coincides with \cite[Thm.\,3.4]{Bludov2025}.

\medskip
\begin{lemma}
	\label{lem: topology of convexity complex}
	Let $V$ be a finite set and let $A\colon \Delta_V\longrightarrow\E$ be an affine map such that $A(V)\subseteq S(\E)$.
	Then
	\[
		\CC_A \ \simeq \
		\begin{cases}
			\  S^{\dim(\spann(A(V)))-1} , & 0\in \relint (\conv (A(V))),    \\
			\   \pt,                      & 0\notin  \relint(\conv (A(V))).
		\end{cases}
	\]
\end{lemma}
\begin{proof}
	There is a sequence of relations
	\[
		\CC_A\overset{\text{Lem. }\ref{lem : nerve}}{\cong}\NN_{\U_A}\overset{\text{Nerve thm.}}{\simeq}\bigcup_{u\in V}\mathscr{B}_{A(u)}=\mathscr{B}\langle A(V)\rangle ,
	\]
	which allows us to use Lemma~\ref{lem :  case 01}, Lemma~\ref{lem :  case 02} and Lemma~\ref{lem :  case 03}, or equivalently \eqref{eq : 00}.

	\medskip
	It remains to show that in the case when $ 0\in \relint\conv (A(V))$ we have
	\begin{equation}
		\label{eq : 01}
		S(\E)-\{y\in S(\E) : (\forall u\in V)\ \langle y,A(u)\rangle=0 \}\ \simeq \ S^{\dim (\spann(A(V)))-1}.
	\end{equation}
	Set $W=\spann(A(V))$, and denote by $W^{\perp}$ its orthogonal complement.
	Then
	\[\{y\in S(\E) : (\forall u\in V)\ \langle y,A(u)\rangle=0 \}\,=\,S(W^{\perp}).\]
	The relation \eqref{eq : 01} follows from the following observations about the (topological) joins:
	\[
		S(\E)\cong S(W\oplus W^{\perp})\cong S(W)*S(W^{\perp})
		\qquad\text{and}\qquad
		X*Y-Y\simeq X.
	\]
	Here $X$ and $Y$ are assumed to be, for example, finite simplicial complexes.
\end{proof}

\medskip
\subsection{Alexander duals of the zero-avoiding complexes}

The Alexander duality~\cite{Alexander1922} is one of the most frequently used classical results of algebraic topology.
In the context of the category of simplicial complexes it can be formulated as follows.

\medskip
Let $\K$ be a simplicial complex on the vertex set $V$.
The simplicial complex
\[
	\K^*:= \{ F\subseteq V :V-F \notin\K\}
\]
is called the {\em Alexander dual} of $\K$.
In particular, $(\K^*)^*=\K$.
Now, the combinatorial version of the Alexander duality can be stated as follows, see for example~\cite[Thm.\,1.1]{BjornerTancer2009}.

\medskip
From this point on $R$ is assumed to be a commutative ring with unit, and $\widetilde{H}$ stands for the reduced singular or simplicial (co)homology with coefficients in $R$-modules.

\medskip
\begin{theorem}
	\label{th : combinatorial Alexander duality}
	Let $\K$ be a simplicial complex on the finite vertex set $V$ which in not the simplex $\Delta_V$.
	Then for every integer $i\in\Z$ and every coefficient $R$-module $M$ there are isomorphisms
	\begin{equation}
		\label{iso : Alexander 01}
		\widetilde{H}_i(\K;M)\cong  \widetilde{H}^{|V|-3-i}(\K^*;M)
		\qquad\text{and}\qquad
		\widetilde{H}_i(\K^*;M)\cong  \widetilde{H}^{|V|-3-i}(\K;M).
	\end{equation}
\end{theorem}

\medskip
In this section we focus our attention on the topology of induced sub-complexes of the Alexander duals of zero-avoiding complexes.
More precisely, let $\E$ be a finite dimensional affine Euclidean space with the fixed structure of the real vector space.
Additionally, let $V$ be a finite set (in $\R^{\infty}$) and let $A\colon \Delta_V\longrightarrow\E$ be an affine map such that $A(V)\subseteq S(\E)$.
Suppose $U\subseteq V$.
The main question we want to answer is: {\em What is the topology of the induced sub-complex
\[
	\CC_A^*[U]:=\{F\subseteq U : F\in \CC_A^*\}
\]
of the zero-avoiding complex $\CC_A$? }

\medskip
The answer to the central question of this section starts with the case  $U=V$, or in other words with the Alexander dual of the zero-avoiding complex.

\medskip
Recall, if $V$ is a finite set and $A\colon \Delta_V\longrightarrow\E$ is an affine map such that $A(V)\subseteq S(\E)$, then
\begin{equation}
	\label{eq : dual}
	\CC_A:=\{U\subseteq V : 0\notin \conv(A(U))\}
	\qquad\text{and}\qquad
	\CC_A^*=   \{ F\subseteq V : 0\in \conv(A(V-F)) \}.
\end{equation}

\begin{lemma}
	\label{lem : topology of dual of complexity complex}
	Let $V$ be a finite set and let $A\colon \Delta_V\longrightarrow\E$ be an affine map such that $A(V)\subseteq S(\E)$.
	For every coefficient $R$-module $M$:
	\begin{compactenum}[\quad \rm (1)]
		\item if $0\notin  \relint (\conv (A(V)))$, then for every integer $i$ holds
		\[	 \widetilde{H}_i(\CC_{A}^*;M) =  0,
		\]
		\item if $0\in  \relint (\conv (A(V)))$, then
		\[
			\widetilde{H}_i(\CC_{A}^*;M) \cong
			\begin{cases}
				M, & i= |V|-\dim (\spann(A(V)))-2 ,    \\
				0, & i\neq |V|-\dim (\spann(A(V)))-2 .
			\end{cases}
		\]
	\end{compactenum}
\end{lemma}
\begin{proof}
	This is a direct consequence of the Alexander duality isomorphism of Theorem~\ref{th : combinatorial Alexander duality} and Lemma~\ref{lem: topology of convexity complex}.
	More precisely, we have that
	\[
		\widetilde{H}_i(\CC_A^*;M)\cong   \widetilde{H}^{|V|-3-i}(\CC_A;M)
		\quad\text{and}\quad
		\CC_A \ \simeq \
		\begin{cases}
			\  S^{\dim(\spann(A(V)))-1} , & 0\in \relint (\conv (A(V))),     \\
			\   \pt,                      & 0\notin  \relint (\conv (A(V))),
		\end{cases}
	\]
	which implies the statement of the lemma.

	\medskip
	It is important to observe that in the case (2) when $|V|-\dim(\spann(A(V)))=1$ we have that the dual  $\CC_A^*$ is empty and so the corresponding non-vanishing homology appears only in dimension $-1$ as claimed.
\end{proof}

\medskip
We continue with the study of induced sub-complexes in several steps depending on the geometry of the affine map $A$ and a subset $U\subseteq V$ on which the sub-complex is considered.

\medskip
Recall, if $V$ is a finite set and $A\colon \Delta_V\longrightarrow\E$ is an affine map such that $A(V)\subseteq S(\E)$ we have defined that
\begin{equation}\label{eq : definition of CC_V^*[U]}
	\CC_A^*[U]:=\{F\subseteq U : F\in \CC_A^*\}=\{F\subseteq U : 0\in \conv(A(V-F)) \}.
\end{equation}

\medskip
\begin{lemma}
	\label{lem : topology of induced sub-complexes 01 }
	Let $V$ be a finite set, $A\colon \Delta_V\longrightarrow\E$ an affine map such that $A(V)\subseteq S(\E)$, and let $U\subseteq V$.
	If $0\in\conv( A(V-U))$, then $\CC_A^*[U]$ is the simplex $\Delta_U$, that is a $(|U|-1)$-dimensional simplex.
	In particular,
	\[
		0\in\conv( A(V-U)) \quad \Longrightarrow \quad\CC_A^*[U]\simeq\pt.
	\]
\end{lemma}
\begin{proof}
	Suppose $U\subseteq V$ and $0\in \conv(A(V-U))$.
	Take an arbitrary $F\subseteq U$.
	Then $V-U\subseteq V-F$, and so $0\in \conv (A(V-U))\subseteq\conv (A (V-F))$.
	From \eqref{eq : dual} follows that $F\in\CC_A^*[U]$.
	Consequently, $\CC_A^*[U]$ is the simplex $\Delta_U$.
\end{proof}

\medskip
From now on assume that $V$ is a finite set, $A\colon \Delta_V\longrightarrow\E$ is an affine map such that $A(V)\subseteq S(\E)$, and that $U\subseteq V$ with  $0\notin\conv (A(V-U))$.
Hence, $U\notin\CC_A^*[U]$, and so $\CC_A^*[U]$ is a sub-complex of the boundary of the $(|U|-1)$-dimensional simplex $\Delta_U$ on the vertex set $U$, which is a $(|U|-2)$-dimensional sphere.
In particular, we get the following fact.

\medskip
\begin{lemma}
	\label{lem : topology of induced sub-complexes 02 }
	Let $V$ be a finite set and let $A\colon \Delta_V\longrightarrow\E$ be an affine map such that $A(V)\subseteq S(\E)$.
	If $U\subseteq V$ and $0\notin\conv (A(V-U))$, then for every integer $i\geq |U|-1$ and every coefficient $R$-module $M$ it holds that:
	\[
		\widetilde{H}_i(\CC_A^*[U];M)\cong 0.
	\]
\end{lemma}

\medskip
Note that Lemma~\ref{lem : topology of induced sub-complexes 01 } allows us to remove the assumption $0\notin\conv (A(V-U))$ from the statement of the previous Lemma~\ref{lem : topology of induced sub-complexes 02 }.
In other words, for every $U\subseteq V$ and every integer $i\geq |U| - 1$ it holds that $ \widetilde{H}_i(\CC_A^*[U];M)\cong 0$.

\medskip
If $0\notin\conv (A(V-U))$ then $\CC_A^*[U]$ is either the boundary of the simplex $\Delta_U$ or a proper
sub-complex of $\partial\Delta_U$.
Hence, the Alexander dual of the simplicial complex $\CC_A^*[U]$ in $\partial\Delta_U$ is the simplicial complex given by:
\begin{equation}
	\label{eq : dual of dual 01}
	(\CC_A^*[U])^*=\{ F\subseteq U : U-F\notin \CC_A^*[U]\} = \{ F\subseteq U : 0\notin \conv (A(F\cup (V-U)))\}.
\end{equation}
Note that in the case when $\CC_A^*[U]=\partial\Delta_U$ the dual $(\CC_A^*[U])^*$ is the empty simplicial complex.
Further, from Lemma~\ref{lem : non-empty condition}, we get that
\begin{equation}
	\label{eq : dual of dual 02}
	(\CC_A^*[U])^*=  \{ F\subseteq U : \mathscr{B}(A(F\cup (V-U)))\neq \emptyset\}= \{ F\subseteq U : \mathscr{B}(A(F))\cap \mathscr{B}(A(V-U))\neq \emptyset\}.
\end{equation}
Hence, $(\CC_{A}^*[U])^*$ can be identified with the nerve of the good open covering
\[\{\mathscr{B}_{A(u)}\cap \mathscr{B}(A(V-U)) : u\in U\}\]
of the open set $\mathscr{B}\langle A(U)\rangle\cap \mathscr{B}(A(V-U))$ of the sphere $S(\E)$.
Now, the  Nerve theorem~\cite[Cor.\,4G.3]{Hatcher2002} evidently implies the following claim.

\medskip
\begin{lemma}
	\label{lem : topology of induced sub-complexes 03 }
	Let $V$ be a finite set and let $A\colon \Delta_V\longrightarrow\E$ be an affine map such that $A(V)\subseteq S(\E)$.
	If $U\subseteq V$ and $0\notin\conv(A(V-U))$, then
	\begin{equation}
		\label{eq : dual of dual 03}
		(\CC_A^*[U])^* \ \simeq \ \mathscr{B}\langle A(U)\rangle\cap \mathscr{B}(A(V-U)).
	\end{equation}
\end{lemma}

\medskip
Next, we apply the previous Lemma~\ref{lem : topology of induced sub-complexes 03 } in several situations, thereby gaining better insight in the topology of the induced sub-complex $\CC_A^*[U]$.
We start with a general observation about the connectivity of the induced sub-complex.

\medskip
\begin{lemma}
	\label{lem : topology of induced sub-complexes 03.5 }
	Let $V$ be a finite set and let $A\colon \Delta_V\longrightarrow\E$ be an affine map such that $A(V)\subseteq S(\E)$.
	If $U\subseteq V$ and $0\notin\conv(A(V-U))$, then for every integer $i\leq |U|-\dim(\E)-2$ and every coefficient $R$-module $M$ it holds that
	\[
		\widetilde{H}_i(\CC_A^*[U];M)\cong 0.
	\]
\end{lemma}
\begin{proof}
	Since $0\notin\conv(A(V-U))$ from Lemma~\ref{lem : topology of induced sub-complexes 03 } we conclude that
	\[
		(\CC_A^*[U])^* \ \simeq \ \mathscr{B}\langle A(U)\rangle\cap \mathscr{B}(A(V-U))\ \subseteq \  \mathscr{B}(A(V-U)) \ \cong \R^{\dim(\E)-1}.
	\]
	Hence, $\widetilde{H}^j((\CC_A^*[U])^*;M)\cong 0$ for all $j\geq \dim(\E)-1$.
	Now, the Alexander duality isomorphism \eqref{iso : Alexander 01} implies that
	\[
		\widetilde{H}_i(\CC_A^*[U];M)\ \cong  \ \widetilde{H}^{|U|-3-i}((\CC_A^*[U])^* ;M)
		\ \cong  \ 0
	\]
	for all $|U|-3-i\geq \dim(\E)-1$, or equivalently for all $i\leq |U|-\dim(\E)-2$, as claimed.
\end{proof}

\medskip
Combining Lemma~\ref{lem : topology of induced sub-complexes 01 } with the previous Lemma~\ref{lem : topology of induced sub-complexes 03.5 } we can remove the condition $0\notin\conv(A(V-U))$ and obtain the following corollary.

\medskip
\begin{corollary}\label{cor : topology of induced sub-complexes 03.5 }
	Let $V$ be a finite set, let $A\colon \Delta_V\longrightarrow\E$ be an affine map such that $A(V)\subseteq S(\E)$, and let $U\subseteq V$.
	For every integer $i\leq |U|-\dim(\E)-2$ and every coefficient $R$-module $M$ it holds that
	\[
		\widetilde{H}_i(\CC_A^*[U];M)\cong 0.
	\]
\end{corollary}

\medskip
The last claim of the corollary is a consequence of the first claim and the Hurewicz isomorphism theorem~\cite[Thm.\,VII.10.7]{Bredon2010}.

\medskip
We continue our analysis by considering the case when $0\in\interior(\conv (A(U)))$ next.

\medskip
\begin{lemma}
	\label{lem : topology of induced sub-complexes 04 }
	Let $V$ be a finite set and let $A\colon \Delta_V\longrightarrow\E$ be an affine map such that $A(V)\subseteq S(\E)$.
	If  $U\subseteq V$, $0\notin\conv(A(V-U))$ and $0\in\interior(\conv (A(U)))$, then for every integer $i$ and every coefficient $R$-module $M$ holds:
	\begin{equation}
		\label{eq : dual of dual 04}
		\widetilde{H}_i(\CC_A^*[U];M)\cong 0.
	\end{equation}
\end{lemma}
\begin{proof}
	Since $0\in\interior(\conv (A(U)))$ using Lemma~\ref{lem : covering of the sphere} we get that $\mathscr{B}\langle A(U)\rangle=S(\E)$.
	The assumption $0\notin\conv(A(V-U))$ allows us to apply Lemma~\ref{lem : topology of induced sub-complexes 04 } as follows:
	\[
		(\CC_A^*[U])^* \ \simeq \ \mathscr{B}\langle A(U)\rangle\cap \mathscr{B}(A(V-U)) \ =\ S(\E)\cap \mathscr{B}(A(V-U))\ =\ \mathscr{B}(A(V-U))\ \simeq \pt.
	\]
	Now, the Alexander duality isomorphism \eqref{iso : Alexander 01} completes the argument.
\end{proof}

\medskip
The next case partially overlaps with the situation from the previous Lemma~\ref{lem : topology of induced sub-complexes 04 }.

\medskip
\begin{lemma}
	\label{lem : topology of induced sub-complexes 05 }
	Let $V$ be a finite set and let $A\colon \Delta_V\longrightarrow\E$ be an affine map such that $A(V)\subseteq S(\E)$.
	If $U\subseteq V$, $0\notin\conv(A(V-U))$ and $\cone(A(V-U))\cap \cone (A(U))\neq \{0\}$, then for every integer $i$ and every coefficient $R$-module $M$ it holds that:
	\begin{equation}
		\label{eq : dual of dual 05}
		\widetilde{H}_i(\CC_A^*[U];M)\cong 0.
	\end{equation}
\end{lemma}
\begin{proof}
	Recall that the assumption $0\notin\conv(A(V-U))$ allows us to apply Lemma~\ref{lem : topology of induced sub-complexes 03 } and get that $(\CC_A^*[U])^*  \simeq  \mathscr{B}\langle A(U)\rangle\cap \mathscr{B}(A(V-U))$.
	We will show that the assumption $\cone(A(V-U))\cap \cone (A(U))\neq \{0\}$ implies that $\mathscr{B}(A(V-U))\subseteq \mathscr{B}\langle A(U)\rangle$.
	Consequently, we will have that
	\[
		(\CC_A^*[U])^*  \simeq  \mathscr{B}\langle A(U)\rangle\cap \mathscr{B}(A(V-U)) =\mathscr{B}(A(V-U))\simeq\pt.
	\]
	Then the Alexander duality isomorphism \eqref{iso : Alexander 01} will complete the argument.

	\medskip
	Let $x\in \mathscr{B}(A(V-U))$, or equivalently $\langle x,A(v)\rangle>0$ for all $v\in V-U$.
	Choose a non-zero vector $z\in \cone(A(V-U))\cap \cone (A(U))\neq \{0\}$.
	Then
	\[
		z=\sum_{v\in V-U}t_v \cdot A(v)=\sum_{u\in U}s_u \cdot A(u)
	\]
	where $t_v\geq 0$, $s_u\geq 0$, and at least one $t_v$ and at least one $s_u$ are strictly positive.
	Next, we have that:
	\[
		\sum_{u\in U}s_u \,\langle x, A(u)\rangle=\langle x,z\rangle = \sum_{v\in V-U}t_v \,\langle x, A(v)\rangle >0.
	\]
	Hence, there is at least one $u\in U$ with the property that $s_u>0$ and $\langle x,A(u)\rangle>0$, and so $x\in \mathscr{B}_{A(u)}\subseteq \mathscr{B}\langle A(U)\rangle$.
	In this way we have verified that $\mathscr{B}(A(V-U))\subseteq \mathscr{B}\langle A(U)\rangle$ and consequently completed the proof of the lemma.
\end{proof}

\medskip
Let $V$ be a finite set,  let $A\colon \Delta_V\longrightarrow\E$ be an affine map such that $A(V)\subseteq S(\E)$, and let $U\subseteq V$.
For the rest of the section we make the following assumptions:
\begin{compactenum}
	\item[(A1)] 	$0\notin\conv(A(V-U))$,
	\item[(A2)] 	$0\notin\interior(\conv (A(U)))$,
	\item[(A3)] 	 $\cone(A(V-U))\cap \cone (A(U))= \{0\}$.
\end{compactenum}
Furthermore, we denote by \[L_U:=\cone(A(U))\cap\cone(-A(U))=\cone(A(U))\cap(-\cone(A(U)))\] the maximal linear subspace of $\E$ contained in $\cone(U)$; for an illustration see Figure~\ref{fig 022}.

\begin{figure}
	\includegraphics[scale=1]{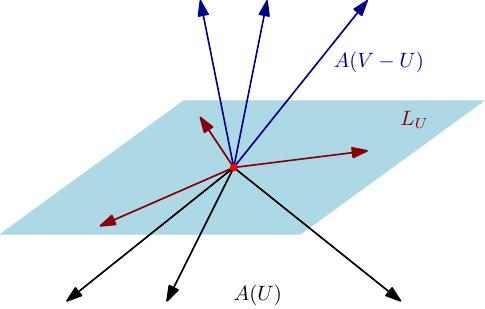}
	\caption{\small An illustration of the set $V$ when the assumptions (A1), (A2) and (A3) hold.}
	\label{fig 022}
\end{figure}

\medskip
In the next lemmas we relate the topology of $\CC_A^*[U]$ with the dimension of the  subspace $L_U$.

\medskip
\begin{lemma}
	\label{lem : topology of induced sub-complexes 06 }
	Let $V$ be a finite set, let $A\colon \Delta_V\longrightarrow\E$ be an affine map such that $A(V)\subseteq S(\E)$, and let $U\subseteq V$.
	If the assumptions {\rm (A1)}, {\rm (A2)} and {\rm (A3)} are satisfied and $\dim(L_U)=\dim(\E)-1$, then for every coefficient $R$-module $M$ it holds that:
	\begin{equation}
		\label{eq : dual of dual 06}
		\widetilde{H}_i(\CC_A^*[U];M)\cong
		\begin{cases}
			M, & i=|U|-\dim(\E)-1=|U|-\dim(L_U)-2,     \\
			0, & i\neq |U|-\dim(\E)-1=|U|-\dim(L_U)-2.
		\end{cases}
	\end{equation}
\end{lemma}
\begin{proof}
	The assumption (A1) implies that $\cone(A(V-U))$ has an apex.
	From the assumption (A3) it follows that the hyperplane $L_U$ separates the point sets $A(U)$ and $A(V-U)$.
	Chose $x\in S(\E)$ to be the unit normal of $L_U$ such that
	\[
		(\forall v\in V-U)\ \langle x,A(v)\rangle <0
		\qquad\text{and}\qquad
		(\forall u \in U)\ \langle x,A(u)\rangle \geq 0.
	\]
	Consequently, $-x\in \mathscr{B}(A(V-U))$ and $-x\notin \mathscr{B}\langle A(U)\rangle$.
	With this, we can show that
	\[
		\mathscr{B}\langle A(U)\rangle\cap \mathscr{B}(A(V-U))\ = \ \mathscr{B}(A(V-U))-\{-x\}.
	\]

	\medskip
	Indeed, first $\mathscr{B}\langle A(U)\rangle\cap \mathscr{B}(A(V-U))\subseteq  \mathscr{B}(A(V-U))-\{-x\}$ because $-x\notin \mathscr{B}\langle A(U)\rangle$.
	Conversely, take $y\in \mathscr{B}(A(V-U))-\{-x\}$.
	Then $y\neq -x$, and also $y\neq x$ because $-x\in \mathscr{B}(A(V-U))$ implies that $x\notin \mathscr{B}(A(V-U))$.
	Since $0\in L_U$ we have that
	\[
		\sum_{u\in U'} t_u\cdot A(u) = 0
	\]
	for some scalars $t_u>0$, where $U':=U\cap A^{-1}(L_U)\neq \emptyset$.
	Then,
	\[
		0=\langle 0, y\rangle = \sum_{u\in U'} t_u\cdot \langle A(u),y\rangle.
	\]
	Since $y\neq -x$ and $y\neq x$ there exists at least one $u'\in U'\subseteq U$ such that $\langle A(u'),y\rangle>0$.
	Thus, $y\in \mathscr{B}_{A(u')}\subseteq \mathscr{B}\langle A(U)\rangle$ and so $y\in  \mathscr{B}\langle A(U)\rangle\cap \mathscr{B}(A(V-U))$ as desired.

	\medskip
	Now, the assumption (A1) with Lemma~\ref{lem : topology of induced sub-complexes 03 } implies that
	\[
		(\CC_A^*[U])^* \ \simeq \ \mathscr{B}\langle A(U)\rangle\cap \mathscr{B}(A(V-U)) \ = \ \mathscr{B}(A(V-U))-\{-x\} \ \simeq \ S^{\dim(\E)-2}
	\]
	because $\mathscr{B}(A(V-U))$ is a $(\dim(\E)-1)$-dimensional open ball.
	The Alexander duality isomorphism \eqref{iso : Alexander 01} implies that
	\[
		\widetilde{H}_i(\CC_A^*[U];M)\ \cong  \ \widetilde{H}^{|U|-3-i}((\CC_A^*[U])^* ;M)
		\ \cong  \ \widetilde{H}^{|U|-3-i}(S^{\dim(\E)-2} ;M),
	\]
	as claimed.
\end{proof}

\medskip
Now we consider a different  geometric situation of the subset $U\subseteq V$.

\medskip
\begin{lemma}
	\label{lem : topology of induced sub-complexes 07 }
	Let $V$ be a finite set, let $A\colon \Delta_V\longrightarrow\E$ be an affine map such that $A(V)\subseteq S(\E)$, and let $U\subseteq V$.
	If the assumptions {\rm (A1)}, {\rm (A2)} and {\rm (A3)} are satisfied and additionally $0\in\relint(\conv (A(U)))$, then for every coefficient $R$-module $M$ we have:
	\begin{equation}
		\label{eq : dual of dual 07}
		\widetilde{H}_i(\CC_A^*[U];M)\cong
		\begin{cases}
			M, & i=|U|-\dim(L_U)-2,     \\
			0, & i\neq |U|-\dim(L_U)-2.
		\end{cases}
	\end{equation}
\end{lemma}
\begin{proof}
	As it the previous proofs, the assumption (A1) with Lemma~\ref{lem : topology of induced sub-complexes 03 } implies that
	$
		(\CC_A^*[U])^* \ \simeq \ \mathscr{B}\langle A(U)\rangle\cap \mathscr{B}(A(V-U))
	$.

	\medskip
	It follows from the assumption $0\in\relint(\conv (A(U)))$ that $L_U=\spann(A(U))$.
	Furthermore, using Lemma~\ref{lem :  case 01}, we get that
	$
		S(\E)=\mathscr{B}\langle A(U)\rangle \cup S
	$,
	where
	\[
		S=\{y\in S(\E) : (\forall u\in U)\ \langle y,A(u) \rangle=0 \}
	\]
	is an equatorial sphere defined by ${\spann(A(U))}^{\perp}$.

	\medskip
	Combining these two facts together we obtain
	\begin{multline*}
		(\CC_A^*[U])^* \ \simeq \ \mathscr{B}\langle A(U)\rangle\cap \mathscr{B}(A(V-U))= \\ \big(S(\E)- S\big)\cap \mathscr{B}(A(V-U)) \ = \ \mathscr{B}(A(V-U)) - \big(  \mathscr{B}(A(V-U)) \cap S\big).
	\end{multline*}
	Since $S$ is an equatorial sphere defined by ${\spann(A(U))}^{\perp}$ and $\mathscr{B}(A(V-U))$ is an open $(\dim(\E)-1)$-dimensional ball we conclude that
	\[
		(\CC_A^*[U])^* \ \simeq \ \R^{\dim(\E)-1} - \R^{\dim(\E)-\dim(L_U)-1}  \ \simeq \ S( \R^{\dim(L_U)-1})=S^{\dim(L_U)-2}.
	\]
	Like the proofs of the previous lemmas the Alexander duality isomorphism \eqref{iso : Alexander 01} implies \eqref{eq : dual of dual 07} and completes the proof.
\end{proof}

\medskip
The final lemma completes our description of the topology of the induced sub-complexes of the Alexander dual of the zero-avoiding complex $\CC_A$.
For that let $V$ be a finite set, let $A\colon \Delta_V\longrightarrow\E$ be an affine map such that $A(V)\subseteq S(\E)$, and let $U\subseteq V$.
Set
\[
	N(V,U):=\mathscr{B}(A(V-U))-\mathscr{B}\langle A(U)\rangle \, \subseteq \, \mathscr{B}(A(V-U)) \, \subseteq \, S(\E).
\]
Note that $N(V,U)$ is a closed subset of $\mathscr{B}(A(V-U))$, where $\mathscr{B}(A(V-U))$ is an open $(\dim(\E)-1)$-dimensional disk and also an open subset of the sphere $S(\E)$.

\medskip
\begin{lemma}
	\label{lem : topology of induced sub-complexes 08 }
	Let $V$ be a finite set, let $A\colon \Delta_V\longrightarrow\E$ be an affine map such that $A(V)\subseteq S(\E)$, and let $U\subseteq V$.
	Take $M$ to be a coefficient $R$-module.
	Assume that {\rm (A1)}, {\rm (A2)} and {\rm (A3)} are satisfied and also assume that
	\begin{compactenum}
		\item[\rm (A4)] 	$0\notin\relint(\conv (A(U)))$.
	\end{compactenum}
	If in addition
	\begin{compactenum}[\rm \quad (1)]
		\item  $N(V,U)$ is compact then
		\begin{equation}
			\label{eq : dual of dual 08 - 1}
			\widetilde{H}^i(\CC_A^*[U];M)\cong
			\begin{cases}
				M, & i=|U|-\dim(\E)-1,     \\
				0, & i\neq |U|-\dim(\E)-1,
			\end{cases}
		\end{equation}
		\item $N(V,U)$ is not compact and $i\leq |U|-\dim(\E)-1$ or $i\geq |U|-\dim(L_U)-1$, then
		\begin{equation}
			\label{eq : dual of dual 08 - 2}
			\widetilde{H}^i(\CC_A^*[U];M)= 0.
		\end{equation}
	\end{compactenum}
\end{lemma}
\begin{proof}
	Recall that the assumption (A1) in combination with Lemma~\ref{lem : topology of induced sub-complexes 03 } implies that
	\[
		(\CC_A^*[U])^*  \simeq \mathscr{B}\langle A(U)\rangle\cap \mathscr{B}(A(V-U)) .
	\]
	Further,  $N(V,U)\neq\emptyset$ by assumption (A3).
	From the definition of $N(V,U)$ we obtain the partition
	\[
		N(V,U)\cup\big( \mathscr{B}\langle A(U)\rangle\cap \mathscr{B}(A(V-U))\big) \ = \ \mathscr{B}(A(V-U)) \ \cong \ \R^{\dim(\E)-1},
	\]
	where
	\[
		N(V,U)\cap \big( \mathscr{B}\langle A(U)\rangle\cap \mathscr{B}(A(V-U))\big)=\emptyset.
	\]

	\medskip
	Set $U':=U\cap A^{-1}(L_U)$ and note that $L_U=\cone(A(U'))$.
	Observe also that $U'$ can be empty which in particular implies that $L_U=\{0\}$.
	The assumption (A4), that is $0\notin\relint(\conv (A(U)))$, implies, in particular, that $U-U'\neq\emptyset$.

	\medskip
	Now, we have the equality
	\begin{equation}
		\label{eq : description of N}
		N(V,U)=  \mathscr{B}(A(V-U)) \cap \mathscr{D}(\{-A(u) : u\in U-U'\})\cap L_U^{\perp}.
	\end{equation}
	Indeed, $x\in N(V,U)=\mathscr{B}(A(V-U))-\mathscr{B}\langle A(U)\rangle \, \subseteq \, \mathscr{B}(A(V-U))$ if and only if
	\[
		(\forall u\in V-U)\, \langle x,A(u)\rangle >0\qquad\text{and}\qquad (\forall u\in U)\, \langle x,A(u)\rangle \leq 0.
	\]
	Furthermore, in the case when $U'\neq\emptyset$, it follows from $L_U=\cone(A(U'))$ that there exist positive real numbers $t_{u}>0$ for $u\in U'$ such that $\sum_{u\in U'}t_u\cdot A(u)=0$.
	Hence, from
	\[
		0=\langle x,0\rangle = \langle x, \sum_{u\in U'}t_u\cdot A(u)\rangle=\sum_{u\in U'}\langle x,  t_u\cdot A(u)\rangle =t_u\cdot\sum_{u\in U'}\langle x,   A(u)\rangle
	\]
	we can conclude that $\langle x, A(u)\rangle=0$ for all $u\in U'$.
	Thus, the quality \eqref{eq : description of N} holds.

	\medskip
	Note that here the sets $A(V-U)$ and $\{-A(u) : u\in U-U'\}$ belong to at least one open hemisphere of $S(\E)$, according to the assumption (A3).
	Furthermore, the sets $A(V-U)$ and $\{-A(u) : u\in U\}$ belong to at least one closed hemisphere of the sphere $S(\E)$.
	Consequently, $N(V,U)$ is strongly convex and so contractible.
	Here strongly convex means that $N(V,U)$ contains no antipodal points, and additionally for every two points of $N(V,U)$ the small arc of the great circle determined by them is contained in $N(V,U)$.
	For the original definition and basic properties consult~\cite[Sec.\,9.1.\,(S)]{DanzerEtAl1963}.
	In modern terminology strong convexity is typically called geodesic convexity.

	\medskip
	{\bf (1)}
	Assume that $N(V,U)\subseteq \mathscr{B}(A(V-U)) \cong   \R^{\dim(\E)-1}$ is compact.
	Then, using Lemma~\ref{lem : topology of induced sub-complexes 03 } and the Alexander duality isomorphism~\cite[Cor.\,VI.8.6]{Bredon2010}, we get that
	\begin{multline*}
		\widetilde{H}_j ((\CC_A^*[U])^*;M)  \ \cong \
		\widetilde{H}_j( \mathscr{B}\langle A(U)\rangle\cap \mathscr{B}(A(V-U));M) \ \cong  \ \\
		\widetilde{H}_j( \mathscr{B}(A(V-U))-N(V,U);M)   \  \cong \ H^{\dim(\E)-2-j}(N(V,U);M)\  \cong   \\ H^{\dim(\E)-2-j}(\pt;M)	.
	\end{multline*}
	Now the Alexander duality isomorphism from Theorem~\ref{th : combinatorial Alexander duality} yields
	\[
		\widetilde{H}^{|U|-3-j}	(\CC_A^*[U];M)\ \cong \ \widetilde{H}_j ((\CC_A^*[U])^*;M) \ \cong \  H^{\dim(\E)-2-j}(\pt;M).
	\]
	Hence,
	$
		\widetilde{H}^{i}	(\CC_A^*[U])\cong   H^{\dim(\E)+1+i-|U|}(\pt;M)
	$,
	as claimed.

	\medskip
	{\bf (2A)}
	Suppose $N(V,U)=\mathscr{B}(A(V-U))$, and so non-compact.
	Then according to the definition of $N(V,U)$ we have that $\mathscr{B}(A(V-U))\cap\mathscr{B}\langle A(U)\rangle=\emptyset$, implying that the dual $(\CC_A^*[U])^*$ is the empty simplicial complex.
	Consequently, $\CC_A^*[U]=\partial\Delta_U\cong S^{|U|-2}$.
	Thus, we have shown case (2) of the lemma for this special case.

	\medskip
	{\bf (2B)}
	Now, assume that $N(V,U)\subsetneq \mathscr{B}(A(V-U)) \cong   \R^{\dim(\E)-1}$ is a proper, closed, but
	not compact subset of $\mathscr{B}(A(V-U))$.
	Hence, we can use a more general version of Alexander duality~\cite[(8.18),\,p.\,301]{Dold1995} which with  Lemma~\ref{lem : topology of induced sub-complexes 03 } gives the following isomorphisms:
	\begin{multline*}
		\widetilde{H}_j ((\CC_A^*[U])^*;M)  \ \cong \
		\widetilde{H}_j( \mathscr{B}\langle A(U)\rangle\cap \mathscr{B}(A(V-U));M) \ \cong \\
		\widetilde{H}_j( \mathscr{B}(A(V-U))-N(V,U);M) \ \cong \
		\widetilde{H}^{\dim(\E)-2-j}(\widehat{N(V,U)};M),
	\end{multline*}
	where $\widehat{N(V,U)}$ denotes the one-point compactification of $N(V,U)$.
	According to the description \eqref{eq : description of N} we have that $N(V,U)$  is  homeomorphic to a full-dimensional closed convex, but not compact, proper subset of $ \mathscr{B}(A(V-U)) \cap L_U^{\perp}\cong\R^{\dim (L_U^{\perp})-1}$.
	Consequently,
	\[
		\widetilde{H}^{\dim(\E)-2-j}(\widehat{N(V,U)};M)\cong 0
	\]
	for $\dim(\E)-2-j\leq 0$ or for $\dim(\E)-2-j\geq \dim (L_U^{\perp})-1$.
	In other words, for  $j\leq \dim(L_U)-1$ or for $\dim(\E)-2\leq j$.
	Hence, the isomorphism $\widetilde{H}^{|U|-3-j}	(\CC_A^*[U];M) \cong \widetilde{H}^{\dim(\E)-2-j}(\widehat{N(V,U)};M)$, in combination with Corollary~\ref{cor : topology of induced sub-complexes 03.5 }, implies the claim \eqref{eq : dual of dual 08 - 2}.
\end{proof}

\medskip
It is important to see that the Universal coefficient theorem,~\cite[Cor.\,V.7.2]{Bredon2010} or~\cite[Cor.\,53.6]{Munkres1984}, used with appropriate coefficients, implies a similar result for the homology as recorded in the following corollary.

\medskip
\begin{corollary}
	\label{cor : topology of induced sub-complexes 09 }
	Let $V$ be a finite set, let $A\colon \Delta_V\longrightarrow\E$ be an affine map such that $A(V)\subseteq S(\E)$, and let $U\subseteq V$.
	Suppose that $\FF$ is an arbitrary field.
	Assume that {\rm (A1)}, {\rm (A2)}, {\rm (A3)} and  {\rm (A4)} are satisfied.
	If additionally
	\begin{compactenum}[\rm \quad (1)]
		\item  $N(V,U)$ is compact then
		\begin{equation}
			\label{eq : dual of dual 08 - 11}
			\widetilde{H}_i(\CC_A^*[U];\FF)\cong
			\begin{cases}
				\FF, & i=|U|-\dim(\E)-1,     \\
				0,   & i\neq |U|-\dim(\E)-1,
			\end{cases}
		\end{equation}

		\item $N(V,U)$ is not compact and $i\leq |U|-\dim(\E)-1$ or $i\geq |U|-\dim(L_U)-1$, then
		\begin{equation}
			\label{eq : dual of dual 08 - 21}
			\widetilde{H}_i(\CC_A^*[U];\FF)= 0.
		\end{equation}
	\end{compactenum}
\end{corollary}

\bigskip
\section{Covering scheme and proofs of colourful Carath\'eodory theorems}
\label{sec : Covering scheme and proofs of colourful Caratheodory theorems}
\bigskip

In this section we first develop the so-called ``covering scheme'', Lemma~\ref{lem : affine CS/TM scheme 02 caratheodory}, which offers an algorithm for proving different instances of
Problem~\ref{prob : general_colourful_Caratheodory theorem}:
\begin{quote}
	For integers $N\geq 1$ and $d\geq 1$, find all simplicial sub-complexes $\mathcal{K}$ of the simplex $\Delta_N$, with $\vertex(\mathcal{K})=\vertex(\Delta_N)$, and all the families $\mathcal{S}$ of faces of $\Delta_N$ such that the following implication holds: If an affine map $A\colon\Delta_N\longrightarrow\R^d$ satisfies that $0\in A(S)$ for every face $S\in \mathcal{S}$, then there exists a face $J$ of $\mathcal{K}$ with $0\in A(J)$.
\end{quote}
Then, using the covering scheme we give essentially identical proofs to all Theorems \ref{th : colourful_Caratheodory_01}, \ref{th : colourful_Caratheodory_02}, \ref{th : colourful_Caratheodory_03}, and \ref{th : main_result_2}.

\medskip
An additional ingredient of all these proofs will be the homological Nerve theorem of Meshulam \cite[Thm.\,2.1]{Meshulam2001}:
\begin{quote}
	{\bf Theorem 2.1.}  Let $X$ be a connected finite simplicial complex and let $\U:=(X_i : i\in I)$ be a cover of $X$ by  sub-complexes. Additionally, assume that for every non-empty $t$-fold intersection $X_{i_1}\cap \cdots\cap X_{i_t}$ and every $0\leq j\leq k-t+1$ it holds that
	$
		\widetilde{H}_j(X_{i_1}\cap \cdots\cap X_{i_t};\FF)=0
	$,
	where $t\geq 1$. Then for arbitrary field $\FF$ coefficients it holds that
	\begin{compactenum}[\rm \quad (1)]
		\item $\widetilde{H}_j(X;\FF)\cong \widetilde{H}_j(\NN_{\U};\FF)$ for all $0\leq j\leq k$, and
		\item if $\widetilde{H}_{k+1}(\NN_{\U};\FF)\neq 0$, then $\widetilde{H}_{k+1}(X;\FF)\neq 0$.
	\end{compactenum}
\end{quote}

\medskip
For simplicity's sake in the following we say that for $t\geq -1$ a simplicial complex $\K$ is {\em homologically $t$-connected}, with respect to the coefficients in a field $\FF$, if $\widetilde{H}_i(\K;\FF)=0$ for all $-1\leq i\leq k$.

\medskip
To apply the homological Nerve theorem we will have to estimate the homological connectivity of joins of sub-complexes.
For that we will use the K\"unneth formula for joins with coefficients in a field $\FF$ which says that:
\begin{equation}\label{eq:kunneth_for_join}
	\widetilde{H}_k(A*B;\FF)\cong \bigoplus_{a+b=k-1} \widetilde{H}_a(A;\FF)\otimes_{\FF} \widetilde{H}_b(B;\FF),
\end{equation}
where $A$ and $B$ are, for example, finite simplicial complexes.
In particular, we have that if $A$ is homologically $t$-connected and $B$ is homologically $s$-connected, then $A*B$ is homologically $(t+s+2)$-connected.

\medskip
\subsection{Transition into a topological question}
Let $N\geq 1$ and $d\geq 1$ be integers, let $V:=\vertex(\Delta_N)=\{p_0,\dots,p_N\}$, $\mathcal{S}$ a family of faces of the simplex $\Delta_N$ and $\mathcal{K}$ a simplicial complex with $V=\vertex(\mathcal{K})$.
Fix an affine map $A\colon\Delta_N\longrightarrow\R^d$ which has the property that $0\in A(S)$ for every $S\in \mathcal{S}$.

\medskip
To the fixed affine map $A$ we associate two simplicial complexes
\[
	X:=X(A,\mathcal{K}):=\CC_A\qquad\text{and}\qquad Y:=Y(A,\mathcal{K}):=\mathcal{K}.
\]
Here $\CC_A:=\{U\subseteq V : 0\notin \conv(A(U))=A(\conv(U))\}$ denotes the zero-avoiding complex associated to the affine map $A$.
Additionally, set
\[
	Z:=Z(A,\mathcal{K}):=X(A,\mathcal{K})^* * Y(A,\mathcal{K})=X^**Y.
\]
Recall, that $X^*$ denotes the Alexander dual simplicial complex of $X$.
The vertex set of $Z$ is the disjoint union of two copies of $V$, say
$
	\vertex(Z)=\vertex(X^*) \sqcup \vertex (Y ):= V^{(1)}\sqcup  V^{(2)}
$.
For $p_{j}\in V$ the corresponding elements in $V^{(1)}$ and $V^{(2)}$ will be denoted by $p_{j}^{(1)}$ and $p_{j}^{(2)}$ respectively.
Now, we colour (partition) the vertex set of $Z$ into $N+1$ colour classes as follows:
\[
	\vertex(Z )= V^{(1)}\sqcup  V^{(2)} = \bigsqcup_{ 0\leq j\leq N} C_{j} = \bigsqcup_{ 0\leq j\leq N,}\{p_{j}^{(1)},p_{j}^{(2)}\}.
\]
Here, as expected,  $C_{j}:=\{p_{j}^{(1)},p_{j}^{(2)}\}$.
A face $T$ of the simplicial complex $\Delta_N*\Delta_N$ is called {\em colourful transversal face} if for every $0\leq j\leq N$ it holds that:
\[
	1\leq |  T\cap C_{j}|=| \vertex (T)\cap C_{j}|\leq 2.
\]

\medskip
Now we can give the first key lemma, which can be seen as an analogue of the so-called configuration space test map scheme paradigm.

\medskip
\begin{lemma}\label{lem : affine CS/TM scheme caratheodory}
	Let $N\geq 1$ and $d\geq 1$ be integers, $V:=\vertex(\Delta_N)$, $\mathcal{S}$ a family of faces of the simplex $\Delta_N$ and $\mathcal{K}$ a simplicial complex with $V=\vertex(\mathcal{K})$.
	Suppose that $A\colon\Delta_N\longrightarrow\R^d$ is an affine map with the property that $0\notin \vertex(\Delta_N)$ and that $0\in A(S)$ for every $S\in \mathcal{S}$.

	\smallskip\noindent
	If the simplicial complex $Z=Z(A,\mathcal{K})$ has at least one colourful transversal face, then there exists a face $J$ of $\mathcal{K}$ such that $0\in A(J)$.
\end{lemma}
\begin{proof}
	Let $T=I*J$ be a colourful transversal face of the simplicial complex $Z=X^**Y$ with $I\in X^*=\CC_A^*$ and $J\in Y=\mathcal{K}$.
	Since $X^*$ is not the simplex $\Delta_N$, the definition of a colourful transversal face implies that the face $J$ of $Y$ is non-empty.

	\medskip
	For simplicity's sake, we assume that the faces $I$ and $J$ are on the same vertex set $V=V^{(1)}=V^{(2)}$.
	An abuse of notation which follows should not cause any confusion.

	\medskip
	The colourful transversality of the face $T=I*J$ yields that $V- \vertex(I) \subseteq  \vertex(J)$ where $J\in Y =\mathcal{K}$.
	Since $I\in  X^*=\CC_A^*= \{ F\subseteq V : 0\in \conv(A(V-F)) \}$ we have that
	\[
		0\in \conv(A(V- \vertex(I)))\subseteq \conv (A( \vertex(J)))=A(J).
	\]
	Thus, we have found a face $J$ of $\mathcal{K}$ with the property that $0\in A(J)$.
\end{proof}

\medskip
We proceed with the second key lemma which we also call the {\em covering scheme}.

\medskip
\begin{lemma}[Covering scheme]\label{lem : affine CS/TM scheme 02 caratheodory}
	Let $N\geq 1$ and $d\geq 1$ be integers, $V:=\vertex(\Delta_N)$, $\mathcal{S}$ a family of faces of the simplex $\Delta_N$ and $\mathcal{K}$ a simplicial complex with $V=\vertex(\mathcal{K})$.
	Suppose that $A\colon\Delta_N\longrightarrow\R^d$ is an affine map with the property that $0\notin A(\vertex(\Delta_N))$ and that $0\in A(S)$ for every $S\in \mathcal{S}$.

	\smallskip\noindent
	If for every face $J$ of $\mathcal{K}$ it holds that $0\notin A(J)$, the then the family of induced sub-complexes
	\[
		\{ Z(A,\mathcal{K})[(V^{(1)}\sqcup  V^{(2)}) - C_{j} ]:  0\leq j\leq N \}
	\]
	is a cover of the simplicial complex $Z(A,\mathcal{K})$.
\end{lemma}
\begin{proof}
	It follows from the assumption and Lemma~\ref{lem : affine CS/TM scheme caratheodory} that the simplicial complex $Z(A,\mathcal{K})$ has no colourful transversal faces.
	Now, let $F=I*J$ be a face of $Z(A,\mathcal{K})=X(A,\mathcal{K})^**Y(A,\mathcal{K})$.
	Since $F$ is not a colourful transversal face, there exists an integer $0\leq j\leq N$  with the property that $|F\cap C_{j}|=0$.
	Consequently, $F\in  Z(A,\mathcal{K})[(V^{(1)}\sqcup  V^{(2)}) - C_{j} ]$ and the proof is completed.
\end{proof}

\medskip
Now, as announced, we use Lemma~\ref{lem : affine CS/TM scheme 02 caratheodory} to prove all colourful Caratheodory theorems, except Theorem~\ref{th : colourful_Caratheodory_04}, which has also an analogous proof.

\medskip
Let us first briefly outline the algorithm of all the proofs which follow.

\medskip
Suppose we are given
\begin{compactitem}[\rm \quad --]
	\item integers $d\geq 1$ and $N\geq 1$,
	\item $N$-dimensional simplex $\Delta_N$ with set of vertices $V=\vertex(\Delta_N)$,
	\item $\mathcal{S}$ a family of faces of the simplex $\Delta_N$,
	\item $\mathcal{K}$ a simplicial complex with $V=\vertex(\mathcal{K})$, and
	\item $A\colon\Delta_N\longrightarrow\R^d$ and affine map which has the property that $0\notin A(\vertex(\Delta_N))$ and $0\in A(S)$ for every $S\in \mathcal{S}$.
\end{compactitem}
The goal is to find a face $J$ of $\mathcal{K}$ with the property that $0\in A(J)$.

\medskip
For this, we:
\begin{compactitem}[\rm \quad --]
	\item consider  the zero-avoiding simplicial complex $X:=X(A,\mathcal{K})=\CC_A$ of the affine map $A$, as defined in \eqref{def of zero-avoiding complex}, and the simplicial complex $Y:=Y(A,\mathcal{K})=\mathcal{K}$,
	\item form the simplicial complex $Z:=X^**Y$ with the vertex set equal to the disjoint union of two copies of $V$, that is $\vertex(Z)=\vertex(X(A,\mathcal{K})^*) \sqcup \vertex (Y(A,\mathcal{K})):= V^{(1)}\sqcup  V^{(2)}$ where $V^{(1)}=V^{(2)}=V$,
	\item consider the family $\U :=  \{ Z[(U^{(1)}\sqcup  U^{(2)}]: \emptyset\subsetneq U \subsetneq V,\, |U|=|V|-1\}$
	of induced sub-complexes of the complex $Z(A,\mathcal{K})$, and crucially
	\item \textit{prove that the family $\U$ cannot cover the simplicial complex $Z$},
	that is show  $\bigcup\U\neq  Z$.
\end{compactitem}

\subsection{Proof of Theorem~\ref{th : colourful_Caratheodory_01}}
\label{subsec : colourful Caratheodory's theorem 01}
\medskip

As a first illustration of the methods, we prove the classical colourful Caratheodory theorem~\cite[Thm.\,2.1]{Barany1981}.

\medskip\noindent
{\bf Theorem~\ref{th : colourful_Caratheodory_01}.}
{\em
	Let $N\geq 1$ and $d\geq 1$ be integers, and let $A\colon \Delta_N\longrightarrow\R^d$ be an affine map.

	\smallskip\noindent
	If there exists an integer $r\geq d+1$ and $r$ pairwise disjoint non-empty faces $F_1,\dots, F_{r}$ of the simplex $\Delta_N$ such that
	$
		0\in A(F_1)\cap\cdots\cap A(F_r)
	$,
	then there exists a selection of vertices $p_1\in F_1,\dots, p_r\in F_r$ with the property that $0\in A(F)$ where $F=\conv\{p_1,\dots,p_r\}\subseteq\Delta_N$ is the face of the simplex $\Delta_N$ spanned by the selected vertices  $p_1,\dots,p_r$.
}

\begin{proof}
	This is a proof by contradiction.
	Fix integers $N\geq 1$, $d\geq 1$,  $r\geq d+1$, and disjoint non-empty faces $F_1,\dots, F_{r}$ of the simplex $\Delta_N$.
	Suppose that $A\colon \Delta_N\longrightarrow\R^d$ is an affine map with the property that $0\notin A(\vertex(\Delta_N))$ and that
	$
		0\in A(F_1)\cap\cdots\cap A(F_r)
	$.
	Without loss of generality we can assume that $A(\Delta_N)$ affinely spans $\R^d$, and that
	\[
		0\in \relint( A(F_1))\cap\cdots\cap  \relint(A( F_r)) \subseteq A(F_1)\cap\cdots\cap A(F_r).
	\]

	\medskip
	Assume that for every selection of vertices $p_1\in F_1,\dots, p_r\in F_r$ it holds that $0\notin A(F)$ where $F=\conv\{p_1,\dots,p_r\}\subseteq\Delta_N$ is the face of the simplex $\Delta_N$ spanned by the selected vertices  $p_1,\dots,p_r$.

	\medskip
	Set $\mathcal{K}:=\vertex ({F_1})*\cdots* \vertex ({F_r})\subseteq\Delta_N$, and further
	\[
		X:=X(A,\mathcal{K})=\CC_A,\qquad Y=Y(A,\mathcal{K})=\mathcal{K}, \qquad Z:=Z(A,\mathcal{K})=X^**Y.
	\]

	\medskip
	We first compute the homology $H_*(Z;\Q)$.
	we have that $0\in \interior(\conv (A(V)))$ from the assumptions, and so by
	Lemma~\ref{lem : topology of dual of complexity complex} and get that
	\[
		\widetilde{H}_i(X^*;\Q)\cong \widetilde{H}_i(\CC_A^*;\Q) \cong
		\begin{cases}
			\Q, & i= N-d-1 ,    \\
			0,  & i\neq N-d-1 .
		\end{cases}
	\]
	Furthermore, $Y=\mathcal{K}=\vertex ({F_1})*\cdots* \vertex ({F_r})$ is the $r$-fold join of finite $0$-dimensional simplicial complexes.
	Consequently, $Y$ is $(r-2)$-connected, meaning
	\[
		\widetilde{H}_i(Y;\Q) \cong
		\begin{cases}
			\neq 0, & i= r-1,    \\
			0,      & i\neq r-1.
		\end{cases}
	\]
	Then, the homology K\"unneth formula for joins \eqref{eq:kunneth_for_join} yields
	\begin{equation}
		\label{eq : homology of Z -- C01}
		\widetilde{H}_i(Z;\Q) \cong \widetilde{H}_i(X^**Y;\Q) \cong
		\begin{cases}
			\neq 0, & i= N-d-1+r,      \\
			0,      & i\neq  N-d-1+r .
		\end{cases}
	\end{equation}

	\medskip
	Now, we conclude from the assumption on the affine map $A$ and Lemma~\ref{lem : affine CS/TM scheme 02 caratheodory} that the family of induced sub-complexes
	\[
		\U :=  \{ Z[(V^{(1)}\sqcup  V^{(2)}) - C_{j} ] \,:\,  0\leq j\leq N\}
	\]
	is a cover of $Z$.
	Since $C_{j}=\{p_{j}^{(1)},p_{j}^{(2)}\}$ and $V^{(1)}=V^{(2)}=V$, the elements of the cover can be presented as follows
	\[
		Z[(V^{(1)}\sqcup  V^{(2)}) - C_{j} ]= X^*[V-\{p_{j}\}]*Y[V-\{p_{j}\}].
	\]
	Thus, the nerve $\NN_{\U}$ of the cover $\U$ can be identified with the boundary of the simplex on the set of vertices $V$ of the simplex $\Delta_N$, that is, $\NN_{\U}\cong \partial (\Delta_{N})$.

	\medskip
	So far there is no indication that the nerve $\NN_{\U}$ of the cover $\U$ of $Z$ gives any information about the complex $Z$.
	To obtain a relationship between $\NN_{\U}$ and $Z$ we use the homological Nerve theorem of Meshulam \cite[Thm.\,2.1]{Meshulam2001} and the K\"unneth formula for joins \eqref{eq:kunneth_for_join}.

	\medskip
	Hence, we have to estimate the homological connectivity of all non-empty intersections of the elements of the cover $\U$, these are all the sub-complexes of the form $X^*[U]*Y[U]$ where $\emptyset \subsetneq  U \subsetneq  V$.
	This is done in two steps, depending on whether $0\in\R^d$ belongs to $\conv(A(V-U))$.
	\begin{compactenum}[\quad\rm (1)]
		\item If $0\in\conv(A(V-U))$, then $X^*[U]=\CC_A^*[U]$ is a $(|U|-1)$-dimensional simplex according to Lemma~\ref{lem : topology of induced sub-complexes 01 }. Hence, $X^*[U]*Y[U]$ is contractible and therefore homologically $(k-t+1)$-connected whenever $k-t+1\geq -1$.
		\item If $0\notin\conv(A(V-U))$, then $\widetilde{H}_i(\CC_A^*[U];\Q)\cong 0$ for every  $i\leq |U|-d-2$  by Lemma~\ref{lem : topology of induced sub-complexes 03.5 }. From the assumption that $0\in A(F_1)\cap\cdots\cap A(F_r)$ and assumption that $0\notin\conv(A(V-U))$ we get that
		$U \cap  \vertex(F_j)  \neq\emptyset$,
		for every $1\leq j\leq r$. In particular, $|U|\geq r$. Consequently,
		\[ Y[U]\cong \big( U\cap \vertex({F_1})\big) *\cdots *\big( U\cap \vertex ({F_r})\big) \]
		is at least homologically $(r-2)$-connected. Now, using the K\"unneth formula for joins \eqref{eq:kunneth_for_join} we conclude that $X^*[U]*Y[U]$ is homologically $(|U|+r-d-2)$-connected, or in other words for all $i\leq |U|+r-d-2$:
		\[
			\widetilde{H}_i(X^*[U]*Y[U];\Z)\cong 0.
		\]

	\end{compactenum}
	In summary, for every $\emptyset \subsetneq  U \subsetneq  V$ it holds that $\widetilde{H}_i(X^*[U]*Y[U];\Q)\cong 0$ for all $i\leq |U|+r-d-2$.

	\medskip
	Now, for each  $\emptyset \subsetneq  U \subsetneq  V$ the sub-complex $X^*[U]*Y[U]$ is the intersection of $N+1-|U|$ elements of the cover $\U=\{X^*[U']*Y[U'] : \emptyset\subsetneq U' \subsetneq V,\, |U'|=|V|-1\}$.
	Then the homological Nerve theorem  \cite[Thm.\,2.1]{Meshulam2001} implies that $\widetilde{H}_i(Z;\Q)\cong\widetilde{H}_i(\NN_{\U};\Q)$ for all $i\leq N-1$.
	Since, $\NN_{\U}\cong\partial ( \Delta_{N})\cong S^{N-1}$ we have that
	$\widetilde{H}_i(Z;\Q)\cong 0$
	for $1\leq i\leq N-2$, and more importantly, $\widetilde{H}_{N-1}(Z;\Q)\cong \Q$.
	This is a contradiction with \eqref{eq : homology of Z -- C01}, because we have shown that $\widetilde{H}_{i}(Z;\Q)\cong 0$ for all $i\neq N-d-1+r\geq N$ implying that $\widetilde{H}_{N-1}(Z;\Q)\cong 0$.
	This completes the proof of the classical colourful Carath\'eodory theorem.
\end{proof}

\subsection{Proof of Theorem~\ref{th : colourful_Caratheodory_02}}
\label{subsec : colourful Caratheodory's theorem 02}
\medskip

In this section we give a new proof of Theorem~\ref{th : colourful_Caratheodory_02} relying again on the homological Nerve theorem of Meshulam~\cite[Thm.\,2.1]{Meshulam2001}.
The following claim is proved.

\medskip\noindent
{\bf Theorem~\ref{th : colourful_Caratheodory_02}.}
{\em
Let $N\geq 1$ and $d\geq 1$ be integers, and let $A\colon \Delta_N\longrightarrow\R^d$ be an affine map.

\smallskip\noindent
If there exists an integer $r\geq d+1$ and $r$ pairwise disjoint non-empty faces $F_1,\dots, F_{r}$ of the simplex $\Delta_N$ such that
$
	0\in A(F_1)\cap\cdots\cap A(F_{r-1})
$,
then for an arbitrary vertex $p\in F_r$ there exists a selection of vertices $p_1\in F_1,\dots, p_{r-1}\in F_{r-1}$ with the property that $0\in A(F)$ where $F=\conv\{p_1,\dots,p_{r-1},p\}\subseteq\Delta_N$ is the face of the simplex $\Delta_N$ spanned by the selected vertices $p_1,\dots,p_{r-1},p$.
}

\begin{proof}
	Fix integers $d\geq 1$ and $r\geq d+1$, and an affine map $A\colon \Delta_N\longrightarrow\R^d$.
	Suppose that there exist $r$ pairwise disjoint non-empty faces $F_1,\dots, F_{r}$ of the simplex $\Delta_N$ which have the property that
	$
		0\in A(F_1)\cap\cdots\cap A(F_{r-1})
	$.
	Without loss of generality assume that $\vertex(F_1),\dots ,\vertex(F_{r})$ is a partition of $V:=\vertex(\Delta_N)$ and that $F_{r}=\{p\}$ is just a $0$-dimensional face of $\Delta_N$.
	Additionally, let
	\begin{compactitem}[\rm \quad --]
		\item $\mathcal{S}=\{F_1,\dots, F_{r-1}\}$,
		\item $\mathcal{K}=\vertex(F_1)*\cdots*\vertex(F_{r-1})*\{p\}$, a simplicial complex with $V=\vertex(\Delta_N)=\vertex(\mathcal{K})$,
		\item $X:=X(\mathcal{K},A):=\CC_A$, $Y:=Y(\mathcal{K},A):=\mathcal{K}$,   $Z:=Z(\mathcal{K},A)=X^**Y$, and
		\item $\U := \big\{ X^*[U]*Y[U] \subseteq Z : \emptyset\subsetneq U \subsetneq V,\, |U|=|V|-1\big\}$.
	\end{compactitem}
	According to Lemma~\ref{lem : affine CS/TM scheme 02 caratheodory}, it order to prove the theorem, we need to show that $\U$ cannot cover  $Z$.
	For this we study the homologies $H_{*}(Z;\FF)$ and $H_*(\NN_{\U};\FF)$  where $\FF$ is a fixed arbitrary field.

	\medskip
	First we compute homology of the simplicial complex $Z$.
	Since the simplicial complex $\mathcal{K}=\vertex(F_1)*\cdots*\vertex(F_{r-1})*\{p\}$ is just a cone and so contractible it follows that
	\begin{equation}\label{eq:homology_of_Z_caratheodory}
		\widetilde{H}_i(Z;\FF) \cong \widetilde{H}_i(X^**Y;\FF)=\widetilde{H}_i(X^**\mathcal{K};\FF)\cong 0
	\end{equation}
	for all $i\in\Z$.

	\medskip
	Next, we analyse the homology of simplicial complexes of the form $X^*[U]*Y[U]$ where $\emptyset \subsetneq  U \subsetneq  V$.
	This is done, as before, in two steps, depending on whether $0\in\R^d$ belongs to $\conv(A(V-U))$:
	\begin{compactenum}[\quad\rm (1)]
		\item If $0\in\conv(A(V-U))$, then $X^*[U]=\CC_A^*[U]$ is a $(|U|-1)$-dimensional simplex according to Lemma~\ref{lem : topology of induced sub-complexes 01 }. Hence, $X^*[U]*Y[U]$ is contractible and so $\widetilde{H}_{i}\big(X^*[U]*Y[U];\Q\big)=0$ for all $i\in\Z$.
		\item If $0\notin\conv(A(V-U))$, then $\widetilde{H}_i(\CC_A^*[U];\Q)\cong 0$ for every  $i\leq |U|-d-2$  by Lemma~\ref{lem : topology of induced sub-complexes 03.5 }. From the assumptions $
			0\in A(F_1)\cap\cdots\cap A(F_{r-1})
		$ and  $0\notin\conv(A(V-U))$ follows that
		$U \cap  \vertex(F_j)  \neq\emptyset$,
		for every $1\leq j\leq r-1$. In particular, $|U|\geq r-1$.
		Consequently,
		\[
			Y[U]\cong \big( U\cap \vertex({F_1})\big) *\cdots *\big( U\cap \vertex ({F_{r-1}})\big) *
			\begin{cases}
				\{p\},     & p\in U,    \\
				\emptyset, & p\notin U,
			\end{cases}
		\]
		is at least homologically $(r-3)$-connected.
		Now, using the K\"unneth formula for joins \eqref{eq:kunneth_for_join} we conclude that $X^*[U]*Y[U]$ is homologically $(|U|+r-d-3)$-connected.
		That is, for all $i\leq |U|+r-d-3$:
		\[
			\widetilde{H}_i(X^*[U]*Y[U];\FF)\cong 0.
		\]

	\end{compactenum}
	In summary, for every $\emptyset \subsetneq  U \subsetneq  V$ it holds that $\widetilde{H}_i(X^*[U]*Y[U];\FF)\cong 0$ for all $i\leq |U|+r-d-3$.

	\medskip
	Now, for each  $\emptyset \subsetneq  U \subsetneq  V$ the sub-complex $X^*[U]*Y[U]$ is the intersection of $N+1-|U|$ elements of the cover $\U=\{X^*[U']*Y[U'] : \emptyset\subsetneq U' \subsetneq V,\, |U'|=|V|-1\}$.
	Using the homological Nerve theorem~\cite[Thm.\,2.1]{Meshulam2001} we have that
	\[
		\widetilde{H}_{j}(Z;\FF)\cong \widetilde{H}_{j}(\NN_{\U};\FF)
	\]
	for all $0\leq j\leq N+r-d-3$, where $N+r-d-2\geq N-1$ because $r\geq d+1$.
	Additionally, if $\widetilde{H}_{N+r-d-2}(\NN_{\U};\FF)\neq 0$, then also $\widetilde{H}_{N+r-d-2}(Z;\FF)\neq 0$.
	Since, $\NN_{\U}\cong\partial ( \Delta_{N})\cong S^{N-1}$ we have reached a contradiction with \eqref{eq:homology_of_Z_caratheodory}, because $\widetilde{H}_{i}(Z;\FF)\cong 0$ for all $i\in \Z$ but $\widetilde{H}_{N-1}(\NN_{\U};\FF)\cong \FF$ and $N+r-d-3\geq N-2$.
	This completes the proof of the second version of the classical colourful Carath\'eodory theorem.
\end{proof}

\subsection{Proof of Theorem~\ref{th : colourful_Caratheodory_03}}
\label{subsec : colourful Caratheodory's theorem 03}
\medskip

In this section we reprove the following result using the methods we have developed.

\medskip\noindent
{\bf Theorem~\ref{th : colourful_Caratheodory_03}.}
{\em
	Let $N\geq 1$ be an integer, and let $A\colon \Delta_N\longrightarrow\R^d$ be an affine map.

	\smallskip\noindent
	If there exists an integer $r\geq d+1$ and $r$ pairwise disjoint non-empty faces $F_1,\dots, F_{r}$ of the simplex $\Delta_N$ such that
	$
		0\in \bigcap_{1\leq i<j\leq r} A(F_i\cup F_j)
	$,
	then there exists a selection of vertices $p_1\in F_1,\dots, p_r\in F_r$ with the property that $0\in A(F)$ where $F=\conv\{p_1,\dots,p_r\}\subseteq\Delta_N$ is the face of the simplex $\Delta_N$ spanned by the selected vertices  $p_1,\dots,p_r$.}
\begin{proof}
	Once again we follow the steps of our ``algorithm''.
	Fix integers $d\geq 1$ and $r\geq d+1$, and in addition an affine map $A\colon \Delta_N\longrightarrow\R^d$.
	Suppose that there exist $r$ pairwise disjoint non-empty faces $F_1,\dots, F_{r}$ of the simplex $\Delta_N$ such that, now,
	$
		0\in \bigcap_{1\leq i<j\leq r} A(F_i\cup F_j)
	$.
	Without loss of generality we can assume that $\vertex(F_1),\dots ,\vertex(F_{r})$ is a partition of the vertex set of the simplex $V:=\vertex(\Delta_N)$.
	Also set
	\begin{compactitem}[\rm \quad --]
		\item $\mathcal{S}=\{F_i\cup F_j :  1\leq i<j\leq r\}$,
		\item $\mathcal{K}=\vertex(F_1)*\cdots*\vertex(F_{r})$, a simplicial complex with $V=\vertex(\Delta_N)=\vertex(\mathcal{K})$,
		\item $X:=X(\mathcal{K},A):=\CC_A$, $Y:=Y(\mathcal{K},A):=\mathcal{K}$,   $Z:=Z(\mathcal{K},A)=X^**Y$, and
		\item $\U := \big\{ X^*[U]*Y[U] \subseteq Z : \emptyset\subsetneq U \subsetneq V,\, |U|=|V|-1\big\}$.
	\end{compactitem}
	According to Lemma~\ref{lem : affine CS/TM scheme 02 caratheodory} we need to show that $\U$ cannot cover the simplicial complex $Z$.
	For this we use the $(N-1)$-st homologies $H_{N-1}(Z;\Q)$ and $H_{N-1}(\NN_{\U};\Q)$.

	\medskip
	Due to the assumption $0\in \bigcap_{1\leq i<j\leq r} A(F_i\cup F_j)$ on the affine map $A\colon \Delta_N\longrightarrow\R^d$, we can assume without loss of generality that $A(\Delta_N)$ affinely spans $\R^d$.
	Now we split our proof into two cases depending on whether $0$ is contained in $\interior(\conv (A(V)))$.

	\medskip\noindent\ul{Case 1:} Assume that $0\in \interior(\conv (A(V)))$.
	Then from Lemma~\ref{lem : topology of dual of complexity complex} we  get that
	\[
		\widetilde{H}_i(X^*;\Q)\cong \widetilde{H}_i(\CC_A^*;\Q) \cong
		\begin{cases}
			\Q, & i= N-d-1 ,    \\
			0,  & i\neq N-d-1 .
		\end{cases}
	\]
	Next, $Y=\mathcal{K}=\vertex ({F_1})*\cdots* \vertex ({F_r})$ is $(r-1)$-dimensional and (homologically) $(r-2)$-connected simplicial complex, implying that
	\[
		\widetilde{H}_i(Y;\Q) \cong
		\begin{cases}
			\neq 0, & i= r-1,    \\
			0,      & i\neq r-1.
		\end{cases}
	\]
	Then, the homology K\"unneth formula for joins with field coefficients~\eqref{eq:kunneth_for_join} yields
	\begin{equation}\label{eq : homology of Z -- C03}
		\widetilde{H}_i(Z;\Q) \cong \widetilde{H}_i(X^**Y;\Q) \cong
		\begin{cases}
			\neq 0, & i= N-d-1+r,      \\
			0,      & i\neq  N-d-1+r .
		\end{cases}
	\end{equation}

	\medskip
	Next, like in the previous proofs, we focus on the homology of the induced simplicial complexes $X^*[U]*Y[U]$ where $\emptyset \subsetneq  U \subsetneq  V$.
	Again we are making two steps, depending on whether $0\in\R^d$ belongs to $\conv(A(V-U))$:
	\begin{compactenum}[\quad\rm (1)]
		\item If $0\in\conv(A(V-U))$, then $X^*[U]=\CC_A^*[U]$ is a $(|U|-1)$-dimensional simplex according to Lemma~\ref{lem : topology of induced sub-complexes 01 }. Hence, $X^*[U]*Y[U]$ is contractible and so $\widetilde{H}_{i}\big(X^*[U]*Y[U];\Q\big)=0$ for all $i\in\Z$.
		\item If $0\notin\conv(A(V-U))$, then $\widetilde{H}_i(\CC_A^*[U];\Q)\cong 0$ for every  $i\leq |U|-d-2$  by Lemma~\ref{lem : topology of induced sub-complexes 03.5 }. From the assumptions that $0\notin\conv(A(V-U))$ and $0\in \bigcap_{1\leq i<j\leq r} A(F_i\cup F_j)$ it follows that $U \cap ( \vertex(F_i)\cup \vertex(F_j))   \neq\emptyset$ for all $1\leq i<j\leq r$.
		Hence the set \[\{i : 1\leq i\leq r,\, U \cap\vertex(F_i)\neq\emptyset \}\] has at least $r-1$ elements.
		Consequently,
		\[
			Y[U]\cong \big( U\cap \vertex({F_1})\big) *\cdots *\big( U\cap \vertex ({F_{r}})\big)
		\]
		is at least (homologically) $(r-3)$-connected.
		The K\"unneth formula for joins \eqref{eq:kunneth_for_join} implies that $X^*[U]*Y[U]$ is homologically $(|U|+r-d-3)$-connected, or in other words for all $i\leq |U|+r-d-3$:
		\[
			\widetilde{H}_i(X^*[U]*Y[U];\Z)\cong 0.
		\]
	\end{compactenum}
	Summarising both steps we get that  $\widetilde{H}_i(X^*[U]*Y[U];\Q)\cong 0$ for all $i\leq |U|+r-d-3$ and all $\emptyset \subsetneq  U \subsetneq  V$.

	\medskip
	The final step is identical to the one in the previous section.
	For each  $\emptyset \subsetneq  U \subsetneq  V$ the sub-complex $X^*[U]*Y[U]$ is the intersection of $N+1-|U|$ elements of the cover $\U=\{X^*[U']*Y[U'] : \emptyset\subsetneq U' \subsetneq V,\, |U'|=|V|-1\}$.
	Using the homological Nerve theorem~\cite[Thm.\,2.1]{Meshulam2001} we have that
	\[
		\widetilde{H}_{j}(Z;\Q)\cong \widetilde{H}_{j}(\NN_{\U};\Q)
	\]
	for all $0\leq j\leq N+r-d-3$, where $N+r-d-3\geq N-2$ because $r\geq d+1$.
	Additionally, if $\widetilde{H}_{N+r-d-2}(\NN_{\U};\Q)\neq 0$, then also $\widetilde{H}_{N+r-d-2}(Z;\Q)\neq 0$.
	Since, $\NN_{\U}\cong\partial ( \Delta_{N})\cong S^{N-1}$ we have reached a contradiction with \eqref{eq : homology of Z -- C03}.
	Indeed, $\widetilde{H}_{N-1}(Z;\Q)\cong 0$ but $\widetilde{H}_{N-1}(\NN_{\U};\Q)\cong \Q$ and $N+r-d-3\geq N-2$.

	\medskip\noindent\ul{Case 2:} Assume that $0\notin \interior(\conv (A(V)))$.
	Then  Lemma~\ref{lem : topology of dual of complexity complex} implies that for every integer $i$
	\[
		\widetilde{H}_i(X^*;\Q)\cong \widetilde{H}_i(\CC_A^*;\Q) \cong 0.
	\]
	Like in the previous case , $Y=\mathcal{K}=\vertex ({F_1})*\cdots* \vertex ({F_r})$ is a $(r-1)$-dimensional and (homologically) $(r-2)$-connected simplicial complex, implying that
	\[
		\widetilde{H}_i(Y;\Q) \cong
		\begin{cases}
			\neq 0, & i= r-1,    \\
			0,      & i\neq r-1.
		\end{cases}
	\]
	Now, the homology K\"unneth formula for joins with field coefficients~\eqref{eq:kunneth_for_join} yields that for every integer $i$
	\begin{equation}\label{eq : homology of Z -- C03-2}
		\widetilde{H}_i(Z;\Q) \cong \widetilde{H}_i(X^**Y;\Q) \cong 0.
	\end{equation}

	\medskip
	The homology of the induced simplicial complexes $X^*[U]*Y[U]$, where $\emptyset \subsetneq  U \subsetneq  V$, is evaluated in the same was as in the previous case:
	\begin{compactenum}[\quad\rm (1)]
		\item If $0\in\conv(A(V-U))$, then $X^*[U]=\CC_A^*[U]$ is a $(|U|-1)$-dimensional simplex and so $\widetilde{H}_{i}\big(X^*[U]*Y[U];\Q\big)=0$ for all $i\in\Z$.
		\item If $0\notin\conv(A(V-U))$, then as before $\widetilde{H}_i(\CC_A^*[U];\Q)\cong 0$ for every  $i\leq |U|-d-2$. The assumptions  $0\notin\conv(A(V-U))$ and $0\in \bigcap_{1\leq i<j\leq r} A(F_i\cup F_j)$ imply again that $U \cap ( \vertex(F_i)\cup \vertex(F_j))   \neq\emptyset$ for all $1\leq i<j\leq r$.
		Thus, the set \[\{i : 1\leq i\leq r,\, U \cap\vertex(F_i)\neq\emptyset \}\] has at least $r-1$ elements and so
		\[
			Y[U]\cong \big( U\cap \vertex({F_1})\big) *\cdots *\big( U\cap \vertex ({F_{r}})\big)
		\]
		is at least homologically $(r-3)$-connected.
		The K\"unneth formula for joins \eqref{eq:kunneth_for_join} implies that $X^*[U]*Y[U]$ is homologically $(|U|+r-d-3)$-connected, or in other words that for all $i\leq |U|+r-d-3$:
		\[
			\widetilde{H}_i(X^*[U]*Y[U];\Z)\cong 0.
		\]
	\end{compactenum}
	In summary $\widetilde{H}_i(X^*[U]*Y[U];\Q)\cong 0$ for all $i\leq |U|+r-d-3$ and all $\emptyset \subsetneq  U \subsetneq  V$.

	\medskip
	The final step is identical to the previous case with a small alteration.
	Again, for each  $\emptyset \subsetneq  U \subsetneq  V$ the sub-complex $X^*[U]*Y[U]$ is the intersection of $N+1-|U|$ elements of the cover $\U=\{X^*[U']*Y[U'] : \emptyset\subsetneq U' \subsetneq V,\, |U'|=|V|-1\}$.
	The homological Nerve theorem~\cite[Thm.\,2.1]{Meshulam2001} implies that
	\[
		\widetilde{H}_{j}(Z;\Q)\cong \widetilde{H}_{j}(\NN_{\U};\Q)
	\]
	for all $0\leq j\leq N+r-d-3$, where $N+r-d-3\geq N-2$ because $r\geq d+1$.
	In addition, if $\widetilde{H}_{N+r-d-2}(\NN_{\U};\Q)\neq 0$, then also $\widetilde{H}_{N+r-d-2}(Z;\Q)\neq 0$.
	Since, $\NN_{\U}\cong\partial ( \Delta_{N})\cong S^{N-1}$ we obtained a contradiction with \eqref{eq : homology of Z -- C03-2}.
	Indeed, now we have that $\widetilde{H}_{i}(Z;\Q)\cong 0$ for all $i\in \Z$ but $\widetilde{H}_{N-1}(\NN_{\U};\Q)\cong \Q$ and $N+r-d-3\geq N-2$.

	\medskip
	This completes the proof of the third version of the colourful Carath\'eodory theorem.
\end{proof}

\medskip
\subsection{Proof of Theorem~\ref{th : main_result_2}}\label{sec:proof_ main_result_2}
We now prove the central result of this paper almost identically as we have proven the previous versions of the colourful Carath\'eodory theorem.

\medskip\noindent
{\bf Theorem ~\ref{th : main_result_2}.}
{\em
	Let $N\geq 1$, $d\geq 2$ and $r\geq d+1$ be integers, and let $A\colon \Delta_N\longrightarrow\R^d$ be an affine map.
	Suppose that
	\begin{compactenum}[\rm (1)]
		\item $F_1,\dots, F_{r}$ is a collection of pairwise disjoint non-empty faces of the simplex $\Delta_N$ such that
		$
			0\in A(F_1)\cap\cdots\cap A(F_r)
		$, and
		\item $\mathcal{K}$ is a simplicial sub-complex  $L*\vertex(F_3)*\cdots * \vertex(F_r)$ of $\Delta_N$ where $L$ is a path-connected sub-complex of the join complex $\vertex(F_1)*\vertex(F_2)$ with $\vertex(L)=\vertex(F_1)\cup\vertex(F_2)$.
	\end{compactenum}
	Then there exists a face $J$ of $\mathcal{K}$ with $0\in A(J)$.}
\begin{proof}
	Fix integers $N\geq 1$, $d\geq 2$ and $r\geq d+1$, as well as an affine map $A\colon \Delta_N\longrightarrow\R^d$.
	Choose  $r$ pairwise disjoint non-empty faces $F_1,\dots, F_{r}$ of the simplex $\Delta_N$ such that
	$
		0\in A(F_1)\cap\cdots\cap A(F_{r})
	$.
	Like before, without loss of generality we can assume that $\vertex(F_1),\dots ,\vertex(F_{r})$ is a partition of $V:=\vertex(\Delta_N)$.
	Additionally, let $\mathcal{K}$ be a simplicial sub-complex  $L*\vertex(F_3)*\cdots *\vertex(F_r)$ of $\Delta_N$ where $L$ is a path-connected sub-complex of the join complex $\vertex(F_1)*\vertex(F_2)$ with $\vertex(L)=\vertex(F_1)\cup\vertex(F_2)$.
	Now, set
	\begin{compactitem}[\rm \quad --]
		\item $\mathcal{S}=\{F_1,\dots ,F_{r}\}$,
		\item $\mathcal{K}=L*\vertex(F_3)*\cdots *\vertex(F_r)$,
		\item $X:=X(\mathcal{K},A):=\CC_A$, $Y:=Y(\mathcal{K},A):=\mathcal{K}$,   $Z:=Z(\mathcal{K},A)=X^**Y$, and
		\item $\U := \big\{ X^*[U]*Y[U] \subseteq Z : \emptyset\subsetneq U \subsetneq V,\, |U|=|V|-1\big\}$.
	\end{compactitem}
	Like in all previous proofs, according to Lemma~\ref{lem : affine CS/TM scheme 02 caratheodory}, in order to prove the theorem, we need to show that $\U$ cannot cover $Z$.
	Again, we study the homologies $H_{*}(Z;\Q)$ and $H_*(\NN_{\U};\Q)$.

	\medskip
	Thanks to the assumption on the affine map $A\colon \Delta_N\longrightarrow\R^d$ that
	$
		0\in A(F_1)\cap\cdots\cap A(F_{r})
	$, we can assume without loss of generality that $A(\Delta_N)$ affinely spans $\R^d$, and that
	\[
		0\in \relint( A(F_1))\cap\cdots\cap  \relint(A( F_r)) \subseteq A(F_1)\cap\cdots\cap A(F_r).
	\]
	Consequently, $0\in \interior(\conv (A(V)))$ and so, by Lemma~\ref{lem : topology of dual of complexity complex}, we  get that
	\[
		\widetilde{H}_i(X^*;\Q)\cong \widetilde{H}_i(\CC_A^*;\Q) \cong
		\begin{cases}
			\Q, & i= N-d-1 ,    \\
			0,  & i\neq N-d-1 .
		\end{cases}
	\]
	Furthermore, $Y=\mathcal{K}=L*\vertex(F_3)*\cdots *\vertex(F_r)$ is an $(r-1)$-dimensional and (homologically) $(r-2)$-connected simplicial complex, implying that $\widetilde{H}_i(Y;\Q)\cong 0$ for all $i\neq r-1$.
	Consequently, the K\"unneth formula for a join with field coefficients~\eqref{eq:kunneth_for_join} yields
	\begin{equation}\label{eq : homology of Z -- C04}
		\widetilde{H}_i(Z;\Q) \cong \widetilde{H}_i(X^**Y;\Q) \cong 0
	\end{equation}
	for all $i\neq  N+r-d-1$.

	\medskip
	Once again we consider the homology of $X^*[U]*Y[U]$ where $\emptyset \subsetneq  U \subsetneq  V$.
	This is done in two steps:
	\begin{compactenum}[\quad\rm (1)]
		\item If $0\in\conv(A(V-U))$, then $X^*[U]=\CC_A^*[U]$ is a $(|U|-1)$-dimensional simplex according to Lemma~\ref{lem : topology of induced sub-complexes 01 }. Hence, $X^*[U]*Y[U]$ is contractible and so $\widetilde{H}_{i}\big(X^*[U]*Y[U];\Q\big)=0$ for all $i\in\Z$.
		\item If $0\notin\conv(A(V-U))$, then $\widetilde{H}_i(\CC_A^*[U];\Q)\cong 0$ for every  $i\leq |U|-d-2$  by Lemma~\ref{lem : topology of induced sub-complexes 03.5 }. From the assumptions that $0\notin\conv(A(V-U))$ and  $0\in A(F_1)\cap\cdots\cap A(F_{r})$ it follows that $U \cap  \vertex(F_i)\    \neq\emptyset$ for all $1\leq i\leq r$.
		Consequently,
		\[
			Y[U]\cong L[U] *\big( U\cap \vertex ({F_{3}})\big)*\cdots *\big( U\cap \vertex ({F_{r}})\big)
		\]
		is at least (homologically) $(r-3)$-connected.
		The homology criterion for  connectivity~\cite[Thm.\,4.4.1]{Matousek2008} and the K\"unneth formula for joins \eqref{eq:kunneth_for_join} imply that $X^*[U]*Y[U]$ is homologically $(|U|+r-d-3)$-connected.
		This means that for all $i\leq |U|+r-d-3$:
		\[
			\widetilde{H}_i(X^*[U]*Y[U];\Z)\cong 0.
		\]
	\end{compactenum}
	Summarising both steps we get that  $\widetilde{H}_i(X^*[U]*Y[U];\Q)\cong 0$ for all $i\leq |U|+r-d-3$ and all $\emptyset \subsetneq  U \subsetneq  V$.

	\medskip
	The final step coincides word by word with the one in the previous section.
	For each  $\emptyset \subsetneq  U \subsetneq  V$ the sub-complex $X^*[U]*Y[U]$ is the intersection of $N+1-|U|$ elements of the cover $\U=\{X^*[U']*Y[U'] : \emptyset\subsetneq U' \subsetneq V,\, |U'|=|V|-1\}$.
	Using the homological Nerve theorem~\cite[Thm.\,2.1]{Meshulam2001} we have that
	\[
		\widetilde{H}_{j}(Z;\Q)\cong \widetilde{H}_{j}(\NN_{\U};\Q)
	\]
	for all $0\leq j\leq N+r-d-3$, where $N+r-d-3\geq N-2$ because $r\geq d+1$.
	Additionally, if $\widetilde{H}_{N+r-d-2}(\NN_{\U};\Q)\neq 0$, then also $\widetilde{H}_{N+r-d-2}(Z;\Q)\neq 0$.
	Since, $\NN_{\U}\cong\partial ( \Delta_{N})\cong S^{N-1}$ we have reached a contradiction with \eqref{eq : homology of Z -- C04}.
	Indeed, $\widetilde{H}_{i}(Z;\Q)\cong 0$ for all $i\neq  N+r-d-1$ but $\widetilde{H}_{N-1}(\NN_{\U};\Q)\cong \Q$ and $N+r-d-1\geq N$.
	This completes the proof of our main result.
\end{proof}

\providecommand{\href}[2]{#2}


\begin{thebibliography}{10}

\bibitem{Alexander1922}
J.~W. Alexander, \emph{A proof and extension of the {J}ordan-{B}rouwer
  separation theorem}, Trans. Amer. Math. Soc. \textbf{23} (1922), no.~4,
  333--349.

\bibitem{ArochaBaranyBrachoFabilaMontejano2009}
Jorge~L. Arocha, Imre B\'ar\'any, Javier Bracho, Ruy Fabila, and Luis
  Montejano, \emph{Very colorful theorems}, Discrete Comput. Geom. \textbf{42}
  (2009), no.~2, 142--154.

\bibitem{Barany1982}
Imre B\'ar\'any, \emph{A generalization of {C}arath\'eodory's theorem},
  Discrete Math. \textbf{40} (1982), no.~2-3, 141--152.

\bibitem{Barany2021}
\bysame, \emph{Combinatorial convexity}, University Lecture Series, vol.~77,
  American Mathematical Society, Providence, RI, 2021.

\bibitem{BaranyKalai2022}
Imre B\'ar\'any and Gil Kalai, \emph{Helly-type problems}, Bull. Amer. Math.
  Soc. (N.S.) \textbf{59} (2022), no.~4, 471--502.

\bibitem{Barany1981}
Imre B{\'a}r{\'a}ny, Senya~B. Shlosman, and Andr{\'a}s~Sz{\H{u}} cs, \emph{On a
  topological generalization of a theorem of {T}verberg}, J. Lond. Math. Soc.
  \textbf{23} (1981), 158--164.

\bibitem{BjornerTancer2009}
Anders Bj\"{o}rner and Martin Tancer, \emph{Note: {C}ombinatorial {A}lexander
  duality --- a short and elementary proof}, Discrete Comput. Geom. \textbf{42}
  (2009), no.~4, 586--593.

\bibitem{BlagojevicKarasev2025}
Pavle V.~M. Blagojevi\'{c} and Roman~N. Karasev, \emph{Cone colourful
  carath\'eodory theorems}, in preparation, 2025.

\bibitem{Blagojevic2011-2}
Pavle V.~M. Blagojevi\'c, Benjamin Matschke, and G\"unter~M. Ziegler,
  \emph{Optimal bounds for a colorful {T}verberg--{V}re\'cica type problem},
  Advances in Math. \textbf{226} (2011), 5198--5215.

\bibitem{Blagojevic2009}
\bysame, \emph{Optimal bounds for the colored {T}verberg problem}, J. Eur.
  Math. Soc. (JEMS) \textbf{17} (2015), 739--754.

\bibitem{BlagojevicZiegler2017}
Pavle V.~M. Blagojevi\'{c} and G\"{u}nter~M. Ziegler, \emph{Beyond the
  {B}orsuk-{U}lam theorem: the topological {T}verberg story}, A journey through
  discrete mathematics, Springer, Cham, 2017, pp.~273--341.

\bibitem{Bludov2025}
Mikhail~V. Bludov, \emph{{Balanced sets and homotopy invariant of covers}},
  arXiv:2501.05802 (2024), 19.

\bibitem{Bredon2010}
Glen~E. Bredon, \emph{Topology and {G}eometry}, Graduate Texts in Math., vol.
  139, Springer, New York, 1993.

\bibitem{DanzerEtAl1963}
Ludwig Danzer, Branko Gr\"{u}nbaum, and Victor Klee, \emph{Helly's theorem and
  its relatives}, Proc. {S}ympos. {P}ure {M}ath., {V}ol. {VII}, Amer. Math.
  Soc., Providence, RI, 1963, pp.~101--180.

\bibitem{Dold1995}
Albrecht Dold, \emph{Lectures on algebraic topology}, Classics in Mathematics,
  Springer-Verlag, Berlin, 1995, Reprint of the 1972 edition.

\bibitem{Gruber2007}
Peter~M. Gruber, \emph{Convex and discrete geometry}, Grundlehren der
  mathematischen Wissenschaften, vol. 336, Springer, Berlin, 2007.

\bibitem{Hatcher2002}
Allen Hatcher, \emph{Algebraic topology}, Cambridge University Press,
  Cambridge, 2002.

\bibitem{HolmsenPachTverberg2008}
Andreas~F. Holmsen, J\'anos Pach, and Helge Tverberg, \emph{Points surrounding
  the origin}, Combinatorica \textbf{28} (2008), no.~6, 633--644.

\bibitem{KalaiMeshulam2005}
Gil Kalai and Roy Meshulam, \emph{A topological colorful {H}elly theorem}, Adv.
  Math. \textbf{191} (2005), no.~2, 305--311.

\bibitem{Matousek2008}
Ji\v{r}\'{\i} Matou\v{s}ek, \emph{{Using the {Borsuk--Ulam} Theorem. {L}ectures
  on Topological Methods in Combinatorics and Geometry}}, Universitext,
  Springer-Verlag, Heidelberg, 2003, second corrected printing 2008.

\bibitem{McMullenShephard}
Peter McMullen and Geoffrey~C. Shephard, \emph{Convex polytopes and the upper
  bound conjecture}, London Mathematical Society Lecture Note Series, vol.~3,
  Cambridge University Press, London-New York, 1971.

\bibitem{Meshulam2001}
Roy Meshulam, \emph{The clique complex and hypergraph matching}, Combinatorica
  \textbf{21} (2001), no.~1, 89--94.

\bibitem{Munkres1984}
James~R. Munkres, \emph{Elements of algebraic topology}, Perseus Books, 1984.

\bibitem{Sarkaria1992}
Karanbir~Singh Sarkaria, \emph{Tverberg's theorem via number fields}, Israel J.
  Math. \textbf{79} (1992), no.~2-3, 317--320.

\bibitem{Tverberg1966}
Helge Tverberg, \emph{A generalization of {R}adon's theorem}, J. Lond. Math.
  Soc. \textbf{41} (1966), 123--128.

\end{thebibliography}
\end{document}